\documentclass{amsart}
\usepackage[utf8]{inputenc}
\usepackage[T1]{fontenc}
\usepackage{lmodern}
\usepackage{amsmath}
\usepackage{amssymb}
\usepackage{mathtools}
\usepackage{latexsym}
\usepackage[lite]{amsrefs}
\usepackage{nicefrac}
\usepackage{microtype}
\usepackage{color}
\usepackage{tikz-cd}
\usepackage{enumitem} 
\setlist[enumerate,1]{label=\textup{(\arabic*)}}

\usepackage[all]{xy}
\newdir{ >}{{}*!/-5pt/@{>}}

\usepackage[pdftitle={Analytification functors and homotopy epimorphisms},
pdfauthor={Oren Ben-Bassat and Devarshi Mukherjee},
pdfsubject={Mathematics}
]{hyperref}

\BibSpec{book}{%
  +{}  {\PrintPrimary}                {transition}
  +{,} { \textit}                     {title}
  +{.} { }                            {part}
  +{:} { \textit}                     {subtitle}
  +{,} { \PrintEdition}               {edition}
  +{}  { \PrintEditorsB}              {editor}
  +{,} { \PrintTranslatorsC}          {translator}
  +{,} { \PrintContributions}         {contribution}
  +{,} { }                            {series}
  +{,} { \voltext}                    {volume}
  +{,} { }                            {publisher}
  +{,} { }                            {organization}
  +{,} { }                            {address}
  +{,} { \PrintDateB}                 {date}
  +{,} { }                            {status}
  +{}  { \parenthesize}               {language}
  +{}  { \PrintTranslation}           {translation}
  +{;} { \PrintReprint}               {reprint}
  +{.} { }                            {note}
  +{.} {}                             {transition}
  +{} { \PrintDOI}                   {doi}
  +{} { available at \url}            {eprint}
  +{}  {\SentenceSpace \PrintReviews} {review}
}

\renewcommand*{\PrintDOI}[1]{\href{http://dx.doi.org/\detokenize{#1}}{doi: \detokenize{#1}}}

\newcommand{\comment}[1]{}  

\theoremstyle{plain}
\newtheorem{theorem}{Theorem}[section]
\newtheorem*{theorem*}{Theorem}
\newtheorem*{corollary*}{Corollary}
\newtheorem{lemma}[theorem]{Lemma}
\newtheorem{corollary}[theorem]{Corollary}

\newtheorem{proposition}[theorem]{Proposition}
\theoremstyle{remark}
\newtheorem{remark}[theorem]{Remark}
\theoremstyle{definition}
\newtheorem{definition}[theorem]{Definition}
\newtheorem{example}[theorem]{Example}
\numberwithin{theorem}{section}

\newcommand\N{\mathbb N}
\newcommand\Q{\mathbb Q}
\newcommand\R{\mathbb R}
\newcommand\Z{\mathbb Z}

\newcommand{\coma}{\widehat}
\newcommand{\comb}{\overbracket[.7pt][1.4pt]}

\newcommand*{\Fil}{\mathcal F}

\newcommand{\diff}{\mathrm{d}}
\newcommand{\defeq}{\mathrel{:=}} 

\newcommand*{\into}{\rightarrowtail}
\newcommand*{\onto}{\twoheadrightarrow}

\newcommand*{\ling}[1]{#1_\mathrm{lg}}

\DeclarePairedDelimiter{\abs}{\lvert}{\rvert}
\DeclarePairedDelimiterX{\setgiven}[2]{\{}{\}}{#1\,{:}\,\mathopen{}#2}

\newcommand\hot{\mathbin{\comb{\otimes}}}

\DeclareMathOperator{\Tor}{Tor}

\DeclareMathOperator{\HA}{HA}

\newcommand{\Hom}{\mathsf{Hom}}
\newcommand{\HOM}{\mathsf{HOM}}

\newcommand*{\an}{\mathrm{an}}
\newcommand{\Mod}{\mathsf{Mod}}

\newcommand{\fC}{\mathsf{C}}

\newcommand{\op}{\mathrm{op}}
\newcommand{\tf}{\mathrm{tf}}

\newcommand{\dvr}{V}
\newcommand{\dvgen}{\pi}
\newcommand{\dvf}{F}
\newcommand{\resf}{\mathbb F}

\DeclareMathOperator{\Ext}{Ext}
\DeclareMathOperator{\HP}{HP}
\DeclareMathOperator{\HH}{HH}

\begin{document}
\title{Homological algebra for nonarchimedean completions of group algebras}

\begin{abstract}
Let \(\dvr\) be a complete discrete valuation ring, and let \(G\) be either a word-hyperbolic group or a reductive \(p\)-adic group. We prove that the canonical morphism \(\dvr[G] \to \dvr[G]^\dagger\) from the group algebra to its dagger completion is an isocohomological embedding. This allows us to deduce homological properties about the algebra \(\dvr[G]^\dagger\) in terms of the relatively simpler group algebra of \(G\). These properties are then used towards Hochschild and cyclic homology computations. 
\end{abstract}

\author{Devarshi Mukherjee}

\address{Dep. Matemática-IMAS\\
 FCEyN-UBA, Ciudad Universitaria Pab 1\\
 1428 Buenos Aires\\ 
Argentina}

\email{dmukherjee@dm.uba.ar}

\thanks{The author thanks Guillermo Corti\~nas and Ralf Meyer for helpful discussions, and the hospitality of the IMAS, where this research was conducted. The author was funded by a Feodor Lynen Fellowship of the Alexander von Humboldt Foundation.}

\maketitle

\section{Introduction}

Let \(\dvr\) be a complete discrete valuation ring with uniformiser \(\dvgen\), fraction field \(\dvf\) and residue field \(\resf\).  In this article, we continue our program of noncommutative geometry in mixed characteristic, building on \cites{Cortinas-Cuntz-Meyer-Tamme:Nonarchimedean, Cortinas-Meyer-Mukherjee:NAHA, Meyer-Mukherjee:HA}. These papers have so far dealt with the definitions and properties of appropriate cyclic homology and \(K\)-theories, leading to computations that are accessible primarily from the formal properties of these invariants. Our main examples for which we have complete computations of analytic or periodic cyclic homology include algebras of Hochschild cohomological dimension \(1\), that is, \textit{quasi-free} algebras, including smooth curves and Leavitt path algebras in the commutative and noncommutative cases, respectively. The reason why computations of analytic and periodic cyclic homology for quasi-free algebras are simpler is because for such algebras, Goodwillie's Theorem implies that these theories coincide with the Cuntz-Quillen \(X\)-complex \cite{cuntz1995cyclic}, which is a very small quotient of the periodic cyclic homology complex. 

This article takes a first step towards cyclic homology computations that do not follow from purely formal properties of cyclic homology theory. Needless to mention, such computations are significantly more difficult, and there is no one-fits-all approach (just as in the archimedean setting). Nevertheless, we describe a relatively formal method that enables us to go reasonably far in the computation of periodic cyclic homology for completed \textit{group algebras} of two large classes of groups, namely word-hyperbolic groups and reductive \(p\)-adic groups. The main motivation to study the cyclic homology of group algebras for such groups is that these are precisely the groups for which one can say something meaningful in the context of the Baum-Connes or Farrel-Jones conjectures. Cyclic homology and \(K\)-theoretic tools have previously been used in the archimedean setting by Puschnigg to prove the Kadison-Kaplansky conjecture \cite{puschnigg2002kadison} for word-hyperbolic groups. The ultimate aim of this part our program is to prove the Kaplansky conjecture for the group ring \(\resf[G]\) of a countable group, using completions of the lifting \(\dvr[G]\), and thereafter Puschnigg's techniques involving local cyclic homology and analytic \(K\)-theory, developed in \cites{Meyer-Mukherjee:HL, mukherjee2022nonarchimedean}.

Hochschild homology computations of a (unital) algebra \(A\) typically rely on the knowledge of an explicit (and short) projective bimodule resolution \(P_\bullet \to A \to 0\). For group rings \(\dvr[G]\), such resolutions arise from geometric properties of the underlying group \(G\). Unfortunately, however, the availability of such resolutions does not supply adequate information to compute the Hochschild homology of a \emph{completion} \(A \to B\) of such an algebra. For instance, the ideal candidate for a projective \(B\)-bimodule resolution \(B \hot_A P_\bullet \hot_A B \to B\) may fail to be exact in the sense that there is a bounded \(\dvr\)-linear contracting homotopy. Note that this notion of exactness does not follow from the flatness of \(B\) over \(A\). Following \cite{meyer2004embeddings}, we call a morphism \(f \colon A \to B\) \textit{isocohomological} if the induced complex \(B \hot_A P_\bullet \hot_A B \to B\) is exact. The completion most suited to homotopy invariant cohomology theories in mixed characteristic is the \textit{dagger completion} \(A^\dagger\) of a complete, torsionfree bornological \(\dvr\)-algebra. With this notation in place, one of our main results in this article is the following:

\begin{theorem*}
Let \(G\) be a word-hyperbolic group or a reductive \(p\)-adic group. Then the canonical map \(\dvr[G] \to \dvr[G]^\dagger\) is an isocohomological embedding. 
\end{theorem*}
 
The method of using isocohomological embeddings has previously been used by Meyer \cites{Meyer:Combable, meyer2004embeddings, meyer2006homological} and Solleveld  \cite{solleveld2009periodic} in the archimedean case towards group cohomology and cyclic homology computations for the same class of group algebras. They have also recently appeared in analytic geometry (see \cites{kazhdan2022some,ben2022analytification,aristov2020open}), wherein the notion goes by a plethora of different names, such as homotopy epimorphisms, homological epimorphisms, pseudoflat epimorphisms, etc. A morphism \(A \to B\) being isocohomological has several interesting formal consequences: for instance, the \(\Ext\) and \(\Tor\)-groups relative to \(A\) and \(B\) coincide; if \(A\) has finite projective bidimension, so does \(B\). The latter property has the following consequence for periodic cyclic homology:

\begin{theorem*}
Let \(G\) be word-hyperbolic or a reductive \(p\)-adic group. Then \[\mathsf{HP}(\dvr[G]^\dagger \otimes \dvf) \simeq X^{(n)}(\dvr[G]^\dagger \otimes \dvf)\] for some \(n\). Here \(X^{(n)}\) denotes the quotient of the periodic cyclic complex by the \(n\)-th Hodge filtration.
\end{theorem*} 

In general, the simplification to the \(X^{(n)}\)-complex is the limit of what one can achieve using solely the isocohomological property of a morphism. One then directly compares the \(X^{(n)}\)-complex of the completed group algebra with the incomplete group algebra as in \cite{puschnigg2002kadison}. This will be taken up in the sequel. The projective dimension based methods do however allow us to perform the following complete computations:


\begin{theorem*}
\begin{enumerate}
\item Let \(\abs{\lambda} = 1\) and \(A_\lambda\) denote the \(p\)-adic noncommutative torus, defined as the \(\dvr\)-algebra generated by \(U_1^{\pm 1}\), \(U_2^{\pm 1}\) under the relations \(U_1 U_2 = \lambda U_2 U_1\). Then \[\HA_*(A_\lambda) \cong \HP_*(A_\lambda \otimes \dvf) \cong \dvf^2 \oplus \dvf^2[1]\] for \(* = 0,1\).
\item For \(G = \varprojlim G_n\) a countable profinite group and \(\Lambda(G) = \varprojlim_n \dvr[G_n]\) the pro-Iwasawa algebra. Then we have the following weak equivalences \[\mathsf{HP}(\Lambda(G) \otimes \dvf) \simeq \varprojlim_n \mathsf{HP}(\dvf[G_n]) \simeq \varprojlim_n (\dvf[G_n]/[-,-]).\] 
\end{enumerate}  
\end{theorem*}

This article is organised as follows. In Section \ref{section:prelim}, we describe the sense in which we do homological algebra, following \cite{Meyer:HLHA}. Concretely, starting with a closed symmetric monoidal category \((\fC, \otimes, k)\), one can define module categories \(\mathsf{Mod}_A\) of algebra objects that carry an ``internal" notion of homological algebra: extensions are kernel-cokernel pairs of \(A\)-modules with a \(k\)-linear section. These are the extensions of a Quillen exact category structure on \(\mathsf{Mod}_A\). This exact category has enough projectives and injectives, allowing for the possibility to do homological algebra when (a) the category \(\fC\) is not abelian, which is almost always the case in functional analysis, and (b) one wants to information about periodic cyclic homology for noncommutative algebras. The nature of information one can extract is one of the interesting features of this article, and justifies the use of relative homological algebra. We define the notion of an isocohomological embedding. These are (unital) algebra homomorphisms \(A \to B\) such that for any projective bimodule resolution of \(P_\bullet \to A \to 0\) - relative to the exact category structure on \(\mathsf{Mod}_A\) just mentioned - the induced map \(B \otimes_A P_\bullet \otimes_A B \to B\) is a resolution of \(B\)-bimodules. In Theorem \ref{thm:isocohomological-equivalent}, we describe several equivalent conditions for a morphism to be isocohomological, thereby also deducing some formal consequences. One of these consequences is that if \(A\) has finite Hochschild dimension (which is often known by other methods), then so does \(B\), which implies that the periodic cyclic homology of \(B\) (if \(\fC\) is \(\Q\)-linear) is chain homotopy equivalent to the \(X^{(n)}\)-complex - generalising the result from the quasi-free case to higher dimensions. 

With the consequences of a morphism being isocohomological established, in Section \ref{section:homological-algebra-group} we prove that the canonical maps \(\dvr[G] \to \dvr[G]^\dagger\) are isocohomological embeddings. This uses the same contracting homotopies as in the archimedean case in \cite{Meyer:Combable}, which arise from suitable completions of the Rips complex in geometric group theory. Besides the reduction to the \(X^{(n)}\)-complex, our main applications by way of   complete computations is the recovery of Connes' computation of Hochschild and cyclic homology of a \(p\)-adic version of the noncommutative torus in Section \ref{section:applications-HC}. This \(p\)-adic noncommutative torus will be used in subsequent work on operator theory in the nonarchimedean setting.

\section{Preliminaries}\label{section:prelim}

\subsection{Relative homological algebra}\label{subsection:relative-homological-algebra}

Throughout this paper, we do homological algebra relative to extensions that split in a base symmetric monoidal category. Concretely, let \((\mathsf{C},\otimes)\) be an additive symmetric monoidal category with kernels and cokernels, and suppose further that the cokernels commute with tensor products. In our applications, \(\mathsf{C}\) is also \textit{closed} in the sense that it has an internal \(\Hom\)-functor that is right adjoint to the tensor product functor; in this case the compatibility of \(\Hom\) and \(\otimes\) is automatic. 

Let \(A\) be a unital algebra object in \(\mathsf{C}\) and \(\mathsf{Mod}_A\) the category of unital left \(A\)-modules. Since \(\mathsf{C}\) has kernels and cokernels, so does \(\mathsf{Mod}_A\), and we call a diagram of kernel-cokernel pairs an \textit{extension}. For our applications, want to restrict to those extensions in \(\mathsf{Mod}_A\) that split in the underlying category \(\mathsf{C}\). These are called \textit{semi-split} extensions, and the following is easy to see: 

\begin{lemma}\label{lem:semi-split-exact}
The category \(\mathsf{Mod}_A\) with semi-split extensions as admissible extensions is a Quillen exact category. 
\end{lemma}

The structure of an exact category is sufficient for the basic operations of homological algebra, such as taking projective resolutions and derived categories. We now recall these notions in the generality we require. A chain complex \((X,d)\) in an exact category \(\mathsf{C}\) is called \textit{exact} if the canonical diagram \[\ker(d) \into X \onto \ker(d)\] is degree-wise an admissible extension in \(\mathsf{C}\). A chain map \(f \colon X \to Y\) is called a \textit{quasi-isomorphism} if its mapping cone is exact. The derived category \(\mathsf{D}(\mathsf{C})\) of an exact category \(\mathsf{C}\) is the Verdier localisation of the homotopy category of chain complexes \(\mathsf{HoKom}(\mathsf{C})\) at the quasi-isomorphisms. In the context of our exact category \(\mathsf{Mod}_A\) with semi-split extensions, we have the following:

\begin{lemma}\label{lem:contractible-exact}
A chain complex is semi-split exact if and only if it is contractible in \(\mathsf{Kom}(\mathsf{C})\).
\end{lemma}

A \textit{resolution} of \(M \in \mathsf{Mod}_A\) is a chain complex \((P_n,d_n)\) with an \(A\)-module homomorphism \(d_0 \colon P_0 \to M\) such that the complex \(P_\bullet \overset{d_0}\to M \to 0\) is semi-split exact. An object \(P\) is called \textit{projective} if \(\mathsf{Hom}_A(P,-) \colon \mathsf{Mod}_A \to \mathsf{Ab}\) maps a semi-split extension to a short exact sequence of abelian groups. An object in \(\mathsf{Mod}_A\) is said to be \textit{free} if it is of the form \(A \otimes X\) for some \(X \in \mathsf{C}\) with \(A\)-module structure given by \(m_A \otimes \mathsf{id}_A \colon A \otimes A \otimes X \to A \otimes X\), where \(m_A\) is the multiplication of \(A\). Dually, we call an object \(I\) \textit{injective} if \(\mathsf{Hom}_A(-,I)\) preserves semi-split exact sequences. The dual notion of a free module is a \textit{cofree}-module; we say that an \(A\)-module is cofree if it is of the form \(\mathsf{Hom}(A,M)\) for some \(M \in \mathsf{C}\). Finally, a \textit{projective} (respectively, \textit{free}) \textit{resolution} of \(M\) is a resolution \(P_\bullet \to M\), where each \(P_n\) is projective (respectively, free).

\begin{lemma}\label{lem:enough-projectives}
The exact category \(\mathsf{Mod}_A\) has enough projectives and enough injectives. Furthermore, every chain complex \(X \in \mathsf{HoKom}(\mathsf{Mod}_A)\) has a projective and an injective resolution. 
\end{lemma}

\begin{proof}
For any object \(X \in \Mod_A\), the free \(A\)-module \(A \otimes X\) is projective, and the module structure map \(A \otimes X \to X\) is an admissible epimorphism, with section \(u \otimes 1_X \colon X \to A \otimes X\), where \(u \colon 1 \to A\) is the unit of \(A\). This shows that \(\Mod_A\) has enough projectives. Now since \(\fC\) is closed, for all \(X \in \fC\) we have  \(\Hom_{A}(Y, \Hom_A(A,X)) \cong \Hom(Y,X)\) for all \(Y \in \Mod_A\). Specialising to \(Y = X\), we get a map \(X \to \Hom_A(A,X)\); this splits by the map \(\Hom_A(A,X) \to X\) we get by specialising the tensor-Hom adjunction to \(Y = \Hom_A(A,X)\). With the existence of enough projectives and injectives established, the proof that a chain complex in the homotopy category of chain complexes of \(A\)-modules has a projective and an injective resolution is standard. 
\end{proof}

Let \(A^\op\) denote the algebra that is \(A\) as an object of \(\fC\), but with the opposite multiplication, that is, the multiplication of \(A\) composed with the flip map \(A \otimes A \to A \otimes A\) that exchanges the tensor factors.  Then Lemma \ref{lem:enough-projectives} in particular implies that the mapping space bifunctor \[\mathsf{Hom}_A \colon \mathsf{Mod}_{A \otimes B^\op}^\op \times \mathsf{Mod}_{A \otimes C^{\op}} \to \mathsf{Mod}_{B \otimes C^\op}, (M,N) \mapsto \mathsf{Hom}_A(M,N),\] and the tensor product bifunctor \[- \otimes_A - \colon \mathsf{Mod}_{A \otimes B^\op} \times \mathsf{Mod}_{A \otimes C^\op} \to \mathsf{Mod}_{B \otimes C^\op}, \quad (M,N) \mapsto M \otimes_A N,\] have derived functors, which we denote by \(\mathbb{R}\Hom_A\) and \(- \otimes_A^{\mathbb{L}} - \). Let \[\HOM_A \colon \mathsf{HoKom}(\mathsf{Mod}_{A \otimes B^\op})^\op \times \mathsf{HoKom}(\mathsf{Mod}_{A \otimes C^\op}) \to \mathsf{HoKom}(\mathsf{Mod}_{B \otimes C^\op})\] denote the extension of \(\Hom_A\) to the homotopy category of chain complexes. This is the usual mapping complexes in the category of chain complexes. The derived mapping space is explicitly defined as \[\mathbb{R}\Hom_A(X,Y) = \HOM_A(P(X),Y) \cong \HOM_A(X,I(Y))\cong \HOM_A(P(X),I(Y)),\] where \(X \in \mathsf{HoKom}(\Mod_{A \otimes B^\op})^\op\) and \(Y \in \mathsf{HoKom}(\Mod_{A \otimes C^\op})\), and \(P(X)\) and \(I(Y)\) are projective and injective resolutions of \(X\) and \(Y\), respectively. Likewise, the derived tensor product is defined as \[X \otimes_A^{\mathbb{L}} Y \cong P(X) \otimes_A Y \cong X \otimes_A P(Y) \cong P(X) \otimes_A P(Y),\] for complexes \(X \in \mathsf{HoKom}(\mathsf{Mod}_{A \otimes B^\op})\) and \(Y \in \mathsf{HoKom}(\mathsf{Mod}_{A \otimes C^\op})\).   

Now consider a unital algebra homomorphism \(f \colon A \to B\) in \(\fC\). We can then associate to \(f\) four functors as follows. First, we have the extension of scalar functor \(f^* \colon \mathsf{Mod}_B \to \mathsf{Mod}_A\). This functor has a left adjoint given by base change \(f_{!} \colon \mathsf{Mod}_A \to \mathsf{Mod}_B\), \(M \mapsto B \otimes_A M\). Since \(\mathsf{C}\) is closed, this functor also has a right adjoint \(f_{*} \colon \mathsf{Mod}_A \to \mathsf{Mod}_B\), \(M \mapsto \mathsf{Hom}_A(B,M)\). Finally, we have the functor \(f^{!} \colon \mathsf{Mod}_B \to \mathsf{Mod}_A\), \(M \mapsto \mathsf{Hom}_B(B,M)\), which is left adjoint to \(f_{!}\). Notice that the functors \(f^*\) and \(f^{!}\) coincide. 

The functors above are all additive, so they descend to functors between the homotopy categories of chain complexes. Furthermore, the functors \(f^*\) and \(f^{!}\) are exact, using that \(B\) is projective as a left and right \(B\)-module. Consequently, they descend to functors \(f^* \simeq f^{!} \colon \mathsf{D}(B) \to \mathsf{D}(A)\) on derived categories of the exact categories \(\mathsf{Mod}_B\) and \(\mathsf{Mod}_A\).  Their adjoints \(f_{!}\) and \(f_{-}\) need not however be exact, so we need to take their Kan extensions, defined as follows:

\[
\mathbb{L}f_{!}(M) \defeq B \otimes_{A}^{\mathbb{L}} M, \quad \mathbb{R}f_{*}(M) \defeq \mathbb{R}\mathsf{Hom}_A(B,M).
\] These define functors \(\mathbb{L}f_{!}, \mathbb{R}f_{*} \colon \mathsf{D}(A) \to \mathsf{D}(B)\) at the level of derived categories. The relationship between these four functors on derived categories is summarised in the following proposition:

\begin{proposition}\label{prop:six-functor}
 Let \(f \colon A \to B\) be a unital algebra homomorphism between algebras in a symmetric monoidal category \(\fC\). Then the functor \(\mathbb{R}f_{*}\) is right adjoint to \(f^*\), and \(\mathbb{L}f_{!}\) is left adjoint to \(f^{!}\).   
\end{proposition}
\begin{proof}
It is easy to see that the adjunction between the pairs \((f_*, f^*)\) and \((f^{!}, f_{!})\) extend to adjunctions between the homotopy categories of chain complexes. Now let \(P_\bullet \to X \to I_\bullet\) be a projective and an injective resolution of \(X\) an \(A\)-module. Since \(f_*\) preserves injectives, we have \begin{multline*}
\mathbb{R}\mathsf{Hom}_B(N, \mathbb{R}f_*(M)) \cong \mathbb{R}\mathsf{Hom}_B(N, f_*(I(M))) \\
\cong \mathsf{Hom}_B(N,f_*(I(M))) \cong \mathsf{Hom}_A(f^*N, I(M)) \cong \mathbb{R}\mathsf{Hom}_A(f^*N, M),
\end{multline*} which yields the first adjunction. The proof for the second adjunction is similar and uses that \(f_{!}\) preserves projectives.  
\end{proof}

Together with the tensor-Hom adjunction, the functors in Proposition \ref{prop:six-functor} yield a version of the \(6\)-functor formalism in noncommutative geometry. 

The functor \(f^*\) induces a map between derived mapping spaces given explicitly by 
\begin{multline*}
\mathbb{R}\Hom_A(X,Y) \cong \HOM_A(P(X),P(Y)) \to \\
\HOM_B(f^*(P(X)),f^*(P(Y))) \cong \mathbb{R}\Hom_A(f^*X,f^*Y).
\end{multline*} Similarly, we can define maps 
\[
f^{!} \colon \mathbb{R}\mathsf{Hom}_B(M,N)  \to \mathbb{R}\mathsf{Hom}_A(f^{!}(M), f^{!}(N)),
\] \[f_{*} \colon \mathbb{R}\mathsf{Hom}_A(M,N) \to \mathbb{R}\mathsf{Hom}_B(\mathbb{R}f_*(M), \mathbb{R}f_*(N)),\] and \(f_{!} \colon \mathbb{R}\mathsf{Hom}_A(M,N) \to \mathbb{R}\mathsf{Hom}_B(\mathbb{L}f_{!}(M), \mathbb{L}f_{!}(N))\).  We also have a natural map \[f_* \colon f^*(X) \otimes_A^{\mathbb{L}} f^*(Y) \to X \otimes_A^{\mathbb{L}} Y,\] defined by the composition \[f^*(X) \otimes_A^{\mathbb{L}} f^*(Y) \cong P(f^*(X)) \otimes_A P(f^*(Y)) \to P(X) \otimes_A P(Y) \cong X \otimes_A^{\mathbb{L}} Y,\] where \(P\) is the projective resolution functor.

\subsection{Isocohomological embeddings}\label{subsection:isocohomological}

We now recall the notion of an isocohomological embedding, which have gone by different names with minor variations in context. Let \(\mathsf{C}\) be a symmetric monoidal category  with the additional properties of the previous subsection, and let \(f \colon A \to B\) be a unital algebra homomorphism. By the discussion in the previous section, the extension of scalars functor \(f^* \colon \mathsf{Mod}_B \to \mathsf{Mod}_A\) descends to a functor \(\mathsf{D}(B) \to \mathsf{D}(A)\) between derived categories.

\begin{definition}\label{def:isocohomological}
We call an algebra homomorphism \(f \colon A \to B\) between unital algebras \textit{isocohomological} if the induced functor \(\mathsf{D}(B) \to \mathsf{D}(A)\) is fully faithful. 
\end{definition}

Based on Definition \ref{def:isocohomological} alone, it is hard to ascertain when an algebra homomorphism \(f \colon A \to B\) is isocohomological. The following proposition provides several criteria, which can often be checked in practice.

\begin{theorem}\label{thm:isocohomological-equivalent}

Let \(f \colon A \to B\) be a unital algebra homomorphism. The following statements are equivalent:

\begin{enumerate}
\item\label{isoco:1} The map \(f \colon A \to B\) is an isocohomological embedding;
\item\label{isoco:2} For any projective \(A\)-bimodule resolution \(P_\bullet \to A\), the induced chain complex \(B \otimes_A P_\bullet \otimes_A B\) is a projective \(B\)-bimodule resolution of \(B\);
\item\label{isoco:3} The map \(f_* \colon f^*(B) \otimes_A^{\mathbb{L}} f^*(B) \to B \otimes_B^{\mathbb{L}} B \cong B\) is a chain homotopy equivalence, that is, an isomorphism in \(\mathsf{HoKom}(\fC)\); 
\item\label{isoco:4} The map \(f_* \colon f^*(M) \otimes_A^{\mathbb{L}} f^*(N) \to M \otimes_B^{\mathbb{L}} N\) is a chain homotopy equivalence in \(\mathsf{HoKom}(\fC)\) for all \(M \in \mathsf{D}(B^\op)\) and \(N \in \mathsf{D}(B)\);
\item\label{isoco:5} The map \(f^* \colon \mathbb{R}\mathsf{Hom}_B(B,M) \to \mathbb{R}\mathsf{Hom}_A(f^*(B),f^*(M))\) is an isomorphism in \(\mathsf{HoKom}(\fC)\);
\item\label{isoco:6} The map \(f^* \colon \mathbb{R}\Hom_B(M,N) \to \mathbb{R}\Hom_A(f^*(M), f^*(N))\) is an isomorphism in \(\mathsf{HoKom}(\fC)\) for all \(M\), \(N \in \mathsf{D}(B)\);
\item\label{isoco:7} The map \(f_* \colon B \hot_A^{\mathbb{L}} B \to B\) is a chain homotopy equivalence in \(\mathsf{HoKom}(\fC)\). 
\end{enumerate}
\end{theorem}

\begin{proof}
The proof of \cite{meyer2004embeddings}*{Theorem 35} extends to arbitrary closed symmetric monoidal categories; our case being simpler in that we only consider unital algebras and homomorphisms. For clarity, we prove the equivalence between \ref{isoco:2} and \ref{isoco:3}, which are the only characterisations we use in what follows. Let \(P_\bullet \to A\) be a projective \(A\)-bimodule resolution of \(A\). This means that the complex \(P_\bullet \to A \to 0\) is contractible in \(\mathsf{HoKom}(\fC)\), and  that the \(P_\bullet\) are projective as \(A\)-bimodules. Since \(A\) is projective as a right \(A\)-module, we have that \(P_\bullet \to A\) is an isomorphism in \(\mathsf{HoKom}(A^\op)\). This is because projectivity as \(A\)-modules guarantees that the \(\dvf\)-linear sections of the extension \(P_\bullet \to A \to 0\) extend to \(A\)-linear sections, which we denote by \(s\). Tensoring by \(B\) on the right, we get a chain map \(P_\bullet \hot_A B \to A \hot_A B \cong B\) of projective \(A\)-\(B\)-bimodules. This is a resolution because \(s \otimes \mathrm{1}_B\) is a contracting homotopy for the complex \(P_\bullet \hot_A B \to B \to 0\). Consequently, we have \(B \hot_A^{\mathbb{L}} B \cong B \hot_A (P_\bullet \hot_A B)\), which is a (projective) resolution if and only if \(B \hot_A^{\mathbb{L}} B \to B\) is an isomorphism in \(\mathsf{HoKom}(\fC)\). 
\end{proof}

In other words, Theorem \ref{thm:isocohomological-equivalent} says that if an algebra homomorphism \(f \colon A \to B\) is an isocohomological embedding, one can simply read-off of statements \ref{isoco:4} and \ref{isoco:6} to conclude the  weaker statements about homotopy groups \[\Tor_n^{B}(M,N) \cong \Tor_n^A(M,N) \qquad \Ext_B^n(M,N) \cong \Ext_A^n(M,N)\] for all \(n \in \N\). Likewise, one can deduce information about cohomological dimension, which has interesting consequences in cyclic homology:

\begin{corollary}\label{cor:isocohomo-finite}
Let \(f \colon A \to B\) be an isocohomological embedding. Suppose \(A\) has a projective \(A\)-bimodule resolution of length \(k\), then so does \(B\). 
\end{corollary}
\begin{proof}
Use the equivalence between statements \ref{isoco:2} and \ref{isoco:3} in Theorem \ref{thm:isocohomological-equivalent}. 
\end{proof}

To see the significance of Corollary \ref{cor:isocohomo-finite} to cyclic homology, we first recall the notion of a \textit{bimodule connection} in noncommutative geometry. For \(n \geq 0\) and a unital algebra \(A\), we define the \(A\)-bimodule of \textit{noncommutative differential \(n\)-forms} by
\[\Omega^n(A) \defeq 
\begin{cases}
A \quad n = 0\\
\ker(A \otimes A \overset{m}\to A) \quad n = 1 \\
 \Omega^1(A)^{\otimes n} \quad n \geq 2,
\end{cases} 
\] where \(m \colon A \otimes A \to A\) is the multiplication of the algebra. Equipped with the Hochschild differential \(b \colon \Omega^n(A) \to \Omega^{n-1}(A)\), the complex \((\Omega^\bullet(A), b)\) computes the Hochschild homology of \(A\). 

Now suppose \(A\) has a projective \(A\)-bimodule resolution of length \(n\). Then by \cite{Meyer:HLHA}*{Proposition A.92}, there is a morphism \(\nabla \colon \Omega^n(A) \to \Omega^{n+1}(A)\) satisfying \[\nabla \circ \mu_{A,\Omega^n} = \mu_{A,\Omega^{m+1}(A)} \circ (\mathsf{id}_A \otimes \nabla) + d \otimes \mathsf{id}_{\Omega^n(A)} \colon A \otimes \Omega^n(A) \to \Omega^{n+1}(A),\] where \(\mu_{A,-}\) is the left \(A\)-module structure and \(d \colon A \to \Omega^1(A)\) is the universal derivation. In terms of elements, this is nothing but the Leibnitz rule \(\nabla(a \cdot \omega) = a\nabla(\omega) + \omega \diff(a)\) for all \(\omega \in \Omega^n(A)\) and \(a \in A\). In order not to cause clutter in notation, we assume that our symmetric monoidal category has enough elements and use element notation in what follows. Notice that \(\Omega^n(A) \cong A \otimes \bar{A}^{\otimes n}\), where \(\bar{A}\) is the cokernel of the unit map to \(A\), implying that \(\Omega^n(A)\) is a free right (and similarly left) \(A\)-module, so that \(\nabla\) is automatically a right \(A\)-module map. We call the map \(\nabla\) an \textit{\(n\)-connection on \(A\)}. We can extend such a connection in a canonical way \[\nabla \colon \Omega^{n+m}(A)  \to  \Omega^{n+m+1}(A), \quad \omega \cdot \diff x_{n+1} \cdots \diff x_{n+m} \mapsto \nabla(\omega) \diff x_{n+1} \cdots \diff x_{n+m},\] where \(\omega \in \Omega^{n}(A)\), and \(m \geq 0\). Furthermore, it can be checked easily that \[[\nabla,b] = \mathrm{id}_{\Omega^{n+m}(A)}\] for \(m \geq 0\), where \(b_n \colon \Omega^n(A) \to \Omega^{n+1}(A)\) is the Hochschild differential. Consequently, we have a chain homotopy equivalence \[\mathsf{HH}(A) \simeq (\Omega^n(A)/[-,-] \overset{b}\to \Omega^{n-1}(A) \overset{b}\to \dotsc \to \Omega^1(A) \overset{b}\to A)\] between the Hochschild complex and its truncation to degrees at most \(n\).   

\begin{theorem}\label{thm:main-abstract}
Suppose \(A \to B\) is an isocohomological embedding and \(A\) has finite projective bidimension, then \[\mathsf{HH}(B) \simeq (\Omega^n(B)/[-,-] \overset{b}\to \Omega^{n-1}(B) \overset{b}\to \dotsc \to \Omega^1(B) \overset{b}\to B)\] for some \(n \geq 0\). 
\end{theorem}

\begin{proof}
By Corollary \ref{cor:isocohomo-finite}, \(B\) has finite projective bidimension. This implies the existence of an \(n\)-connection for some \(n\), which contracts the Hochschild complex in degrees \(\geq n\). 
\end{proof}

The contractibility of Hochschild homology in higher degrees also has consequences for periodic cyclic homology computations. For this discussion, let \(\fC\) be additionally \(\Q\)-linear. Let \((\mathsf{HP}(A), b+ B)\) denote the periodic cyclic homology complex of \(A\). The subcomplexes \(\mathcal{F}_n(\mathsf{HP}(A)) \defeq   b(\Omega^{n+1}(A)) \times \prod_{k=n+1}^\infty \Omega^k(A)\) are a decreasing filtration on \(\mathsf{HP}(A)\), called the \textit{Hodge filtration}. Its subquotients are \(\Z/2\Z\)-graded chain complexes \[X^{(n)}(A) \defeq (\mathsf{HP}(A)/\Fil_n(\mathsf{HP}(A)), B+b) = (\prod_{k=0}^{n-1} \Omega^k(A) \times \Omega^n(A)/[-,-], b+ B).\] 

\begin{corollary}\label{cor:HP-finite}
Suppose \(A \to B\) is an isocohomological embedding between unital algebras in a \(\Q\)-linear symmetric monoidal category \(\fC\). Suppose \(A\) has finite projective bidimension. Then \(\mathsf{HP}(B) \simeq X^{(n)}(B)\) for some \(n\). 
\end{corollary}

\begin{proof}
The hypotheses imply that there is an \(n\)-connection on \(B\) for some \(n\) by Theorem \ref{thm:main-abstract}. Now use \cite{Meyer:HLHA}*{Theorem A.123}. 
\end{proof}

\section{Homological algebra for completed group algebras}\label{section:homological-algebra-group}

In this section, we specialise to the situation where the base closed, symmetric monoidal category \(\mathsf{C} = \mathsf{CBorn}_{\dvr}\) - the category of complete bornological \(\dvr\)-modules with the completed projective tensor product. Our goal is to show that the canonical map from the group algebra to its \textit{dagger completion} is an isocohomological embedding for certain classes of groups. We briefly recall the definition of the dagger completion of a bornological \(\dvr\)-algebra, and some generalities about such algebras. Let \(R\) be a bornological \(\dvr\)-algebra. For any such algebra, we can define the \textit{linear growth bornology} on \(R\) as the bornology generated by the submodules \(\sum_{n=0}^\infty \dvgen^n S^{n+1}\), where \(S\) is a bounded \(\dvr\)-submodule of \(R\). This turns \(R\) into a \textit{semi-dagger} algebra, that is, the bounded submodules \(S\) of \(\ling{R}\) by construction satisfy the property that \(\sum_{n=0}^\infty \dvgen^n S^{n+1}\) are bounded. Finally, there are universal arrows \[R \to R^\tf, \quad R \to \ling{R}, \quad R \to \comb{R}\] from a bornological algebra to a bornologically torsionfree \(\dvr\)-algebra \(R^\tf \subseteq R \otimes \dvf\), a semi-dagger algebra and a complete bornological \(\dvr\)-algebra. These combine to define a universal algebra \(R \to R^\dagger\) into a complete, bornologically torsionfree \(\dvr\)-algebra, called a \textit{dagger algebra}.

\begin{example}\label{ex:p-adic-dagger}
Any \(\dvgen\)-adically complete Banach \(\dvr\)-algebra with the bornology where all subsets are bounded is a dagger algebra. 
\end{example}

Example \ref{ex:p-adic-dagger} gives several examples of dagger algebras, and also a recipe to produce them. Indeed, if we start with a torsionfree \(\dvr\)-algebra \(R\), then we can take the \(\dvgen\)-adic completion \(\coma{R} = \varprojlim_{n \in \N} R/ \dvgen^n R\) and equip it with the \(\dvgen\)-adic topology. Torsion-freeness implies that the semi-norm corresponding to the \(\dvgen\)-adic topology is a norm, and the resulting algebra is a Banach \(\dvr\)-algebra. 

From the perspective of invariants, however, \(\dvgen\)-adic completions of torsionfree \(\dvr\)-algebras are less well-behaved than the dagger completion \(R^\dagger\) - a problem already witnessed at the simplest example, namely, the polynomial ring \(\dvr[t]\) in one variable. The periodic cyclic homology of \(\dvr[t] \otimes \dvf\) coincides with rigid cohomology, which is infinite-dimensional in degree \(1\), unlike what we expect from the affine line (for a homotopy invariant theory). This is analogous to the bad behaviour of periodic cyclic homology for \(C^*\)-algebras, motivating one to pass to a smooth subalgebra. The situation in the nonarchimedean case is similar; we have canonical morphisms \[R \to R^\dagger \to \coma{R}\] for a torsionfree \(\dvr\)-algebra \(R\) viewed as a bornological \(\dvr\)-algebra with the initial bornology.

\begin{theorem}\cite{Meyer-Mukherjee:HL}\label{theorem:HL-reduction}
Let \(R\) be a torsionfree \(\dvr\)-algebra with the fine bornology. We then have weak equivalences \[\mathbb{HL}(R^\dagger) \simeq \mathbb{HA}(R/\dvgen R) \simeq \mathbb{HL}(\coma{R})\] between the local cyclic homology complexes of the dagger and \(\dvgen\)-adic completions of \(R\) and the analytic cyclic homology of the reduction mod \(\dvgen\). 
\end{theorem}

Now let \(\tilde{\mathbf{K}}^{\an, \dagger}\) denote the overconvergent, stabilised analytic \(K\)-theory spectrum defined in \cite{mukherjee2022nonarchimedean}. This is a version of Weibel's homotopy algebraic \(K\)-theory spectrum that is homotopy invariant for \(\dvr[t]^\dagger\)-homotopies, stable with respect to the \(\dvgen\)-adic completion of the matrix algebra \(\mathbb{M}_\infty\), and excisive for semi-split extensions of complete, torsionfree bornological \(\dvr\)-algebras. Finally, just as in local cyclic homology, we have the following agreement with the reduction mod \(\dvgen\):

\begin{theorem}\cite{mukherjee2022nonarchimedean}\label{theorem:K-reduction}
Let \(R\) be a torsionfree \(\dvr\)-algebra with the fine bornology such that \(R^\dagger \subseteq \coma{R}\). We then have weak equivalences \[\tilde{\mathbf{K}}^{\an, \dagger}(R^\dagger) \simeq \mathbf{KH}(R/\dvgen) \simeq \tilde{\mathbf{K}}^{\an, \dagger}(\coma{R})\] of spectra. 
\end{theorem}

Theorems \ref{theorem:HL-reduction} and \ref{theorem:K-reduction} together imply that for \(K\)-theory and cyclic homology computations, it does not matter which completion we work with - recovering the well-known fact from \(C^*\)-algebras that the appropriate versions of local cyclic homology and \(K\)-theory are invariant under passage to isoradial embeddings \cite{Cuntz-Meyer-Rosenberg}. The additional feature in the nonarchimedean realm is that these invariants only depend on the reduction mod \(\dvgen\) of the algebra.

Having stated the general results on analytic \(K\)-theory and local cyclic homology, we now specialise our discussion to group algebras \(\dvr[G]\). In light of the analogies drawn in this section, the \(\dvgen\)-adically completed group algebra \(\coma{\dvr[G]}\) should be viewed as the nonarchimedean analogue of a \(C^*\)-algebra. In fact, for countable groups, Claussnitzer and Thom represent this algebra on a \(p\)-adic version of a Hilbert space in \cite{claussnitzer2019aspects}, so our analogy is not merely philosophical.  We now describe the dagger completion \(\dvr[G]^\dagger\) of the group ring of a finitely generated discrete group as a bornological algebra, following \cite{Meyer-Mukherjee:Bornological_tf}.   Let \(F\) be a finite generating set of \(G\) and let \(F_n\) be the set of all elements \(g_1\cdots g_k\) with \(g_1,\dotsc, g_k \in F\), and \(k \leq n\). This is an increasing filtration of \(G\) with \(F_0 = \{e\}\) and \(G = \bigcup_{n = 0}^\infty F_n\). For each \(g \in G\), let \(l(g) \in \N\) be the smallest \(n \in \N\) such that \(g \in F_n\). We equip \(\dvr[G]\) with the fine bornology, so every bounded \(\dvr\)-submodule is contained in \(\dvr[F_n]\) for some \(n\). By definition, a submodule of \(\dvr[G]\) has linear growth if and only if it is contained in \(M_n \defeq \sum_{j=0}^\infty \dvgen^j \dvr[F_n]^{j+1}\), for some \(n \in \N\). That is, \[\ling{\dvr[G]} = \varinjlim M_n.\]

To compute its completion, we first observe that \(M_n = \setgiven{\sum_{g \in G} x_g \delta_g}{\nu(x_g) + 1 \geq l(g)/n}\). So its completion is given by \[\coma{M_n} = \setgiven{\sum_{g \in G} x_g \delta_g}{\nu(x_g) + 1 - l(g)/n \overset{l(g) \to \infty}\to \infty}.\] Furthermore, \(\coma{M_n} \subseteq \coma{M_{n+1}}\), so that \(\dvr[G]^\dagger = \varinjlim \coma{M_n}\). More explicitly, we have
\[\dvr[G]^\dagger = \setgiven{\sum_{g \in G} x_g \delta_g}{\nu(x_g) + 1 \geq c \cdot l(g), c> 0},\] with the bornology where a subset of such power series is bounded if and only if the coefficients of all its elements satisfy \(\nu(a_g) + 1 \geq c \cdot l(g)\) for the same positive constant \(c\). We can view elements of \(\dvr[G]^\dagger\) as functions \(f \colon G \to \dvr\) that are of the form \(f = \sum_{g \in G} a_g \delta_g\), where the coefficients satisfy \(\nu(a_g) + 1 \geq c \cdot l(g)\) for some \(c>0\).

\subsection{Resolutions of completed group algebras}\label{subsec:resolutions-trivial}

Let \(X\) be a nonempty set. We define the \textit{free simplicial set} \(S(X)\) generated by \(X\) as follows: let \(S_n(X) = X^{n+1}\). The \(j\)-th face map \(S_n(X) \to S_{n-1}(X)\) omits \(x_j\) and the \(j\)-th degeneracy map \(S_n(X) \to S_{n+1}(X)\) duplicates \(x_j\). Let \(C_n(X,\dvr)\) be the free \(\dvr\)-module generated by the simplicies in \(S_n(X)\).  Equivalently, \(C_n(X,\dvr)\) is the \(\dvr\)-module of functions \(f \colon X^{n+1} \to \dvr\) of finite support. Equip \(C_n(X,\dvr)\) with the fine bornology. This is a chain complex in the category of complete, torsionfree bornological \(\dvr\)-modules with differential \begin{multline*}
\delta_n \colon C_n(X,\dvr) \to C_{n-1}(X,\dvr), \\
 \delta_n(f)(x_0,\dotsc, x_{n-1}) = \sum_{j=0}^n \sum_{y \in X} (-1)^j f(x_0,\dotsc, x_{j-1}, y, x_j,\dotsc, x_{n-1}),
\end{multline*}   called the \textit{bar complex} on \(X\). A related complex is the \textit{reduced bar complex} \(\bar{C}_n(X,\dvr)\), which is the \(\dvr\)-linear span of non-degenerate simplices, that is, tuples \((x_0,\dotsc,x_n) \in X^{n+1}\) such that \(x_i \neq x_{i+1}\) for all \(i\). These are equivalently functions \(f \colon X^{n+1} \to \dvr\) of finite support such that \(f(x_0,\dotsc, x_n) = 0\) whenever \(x_i = x_{i+1}\) for some \(i\).

The bar complex is functorial for set-theoretic maps. Let \(f \colon X \to Y\) be a map of nonempty sets; this induces a morphism of simplicial sets \(S(X) \to S(Y)\) via \(f_*(x_0,\dotsc, x_n) \defeq (f(x_0),\dotsc, f(x_n))\), which extends to a chain map by linearity \(f_* \colon C_n(X,\dvr) \to C_n(Y,\dvr)\) via \[f_*(x_0,\dotsc,x_n) \defeq (f(x_0),\dotsc,f(x_n)),\] for \((x_0,\dotsc,x_n) \in X^{n+1}\). To get an induced chain map at the level of non-degenerate chains on \(G\), we set \[
f^*(x_0,\dotsc,x_n) \defeq 
\begin{cases}
(f(x_0),\dotsc,f(x_n)), \quad f(x_i) \neq f(x_{i+1}), i \in \{0,\dotsc,n-1\} \\
0 \quad \text{ else }
\end{cases}
\] for \((x_0,\dotsc,x_n) \in \bar{C}_n(X,\dvr)\). We augment the complex \((C_n(X,\dvr),\delta_n)\) (respectively, \((\bar{C}_\bullet(X,\dvr), \delta)\) with the map \(\alpha \colon C_0(X,\dvr) = \dvr[X] \to \dvr\), \(\alpha(f) = \sum_{x \in X} f(x)\), where \(\dvr[X]\) is the \(\dvr\)-module of functions of finite support on \(X\).

Let \[\tilde{C}_n(X,\dvr) = \begin{cases}
\ker(\alpha), n = 0 \\
C_n(X,\dvr), n \geq 1.
\end{cases}
\] 

\noindent That is, \(\tilde{C}_\bullet(X,\dvr)\) is the kernel of the obvious extension of the augmentation map \(\alpha \colon C_0(X,\dvr) \to \dvr\) to a chain map \(C_\bullet(X,\dvr) \to \dvr[0]\) by the zero morphisms in nonzero degree. The following standard argument shows that \(\tilde{C}_n(X,\dvr)\) is contractible for any non-empty set \(X\):

\begin{lemma}\label{lem:trivial-complex}
Let \(X\) be a non-empty set. Then the (reduced) bar complex \((C_\bullet(X,\dvr), \delta)\) (respectively, \((\bar{C}_\bullet(X,\dvr), \delta)\)) is  chain homotopy equivalent to \(\dvr[0]\).
\end{lemma}

\begin{proof}
Let \(f,g \colon X \to Y\) be two maps of non-empty sets. The induced chain maps  \(f_*,g_* \colon C_\bullet(X,\dvr) \to C_\bullet(Y,\dvr)\) are chain homotopic via 
\begin{multline*}
H(f,g) \colon C_n(X,\dvr) \to C_{n+1}(Y,\dvr), \\
(x_0, \dotsc, x_n) \mapsto \sum_{j=0}^n (-1)^j (f(x_0),\dotsc, f(x_j), g(x_j),\dotsc, g(x_n)).
\end{multline*}
Applying this in particular to the obvious maps \(X \to * \to X\), we get the desired chain homotopy equivalence between \(C_\bullet(X,\dvr)\) and \(C_\bullet(*,\dvr) \cong \dvr[0]\). The proof for the reduced bar complex is similar, where we use the induced chain maps on nondegenerate simplices instead. 
\end{proof}

Lemma \ref{lem:trivial-complex} says that \(C_\bullet(X,\dvr)\) is a resolution of \(\dvr\) in the category of complete bornological \(\dvr\)-modules. We now consider the case where \(X = G\) is a group, equipping the \(\dvr\)-modules \(C_\bullet(G,\dvr)\) with the diagonal \(G\)-action \[g\cdot(x_0,\dotsc,x_n) \defeq (gx_0,\dotsc, gx_n),\] and \(\dvr\) with the trivial \(G\)-action. The differentials \(\delta_n \colon C_n(G,\dvr) \to C_{n-1}(G,\dvr)\) and the augmentation map \(\dvr[G] \to \dvr\) are \(G\)-equivariant for these actions.  Now consider the \(\dvr\)-submodule \(C_n'(G,\dvr)\) of \(C_n(G,\dvr)\) generated by \((1,g_1,\dotsc,g_n)\). We then have a \(\dvr[G]\)-module isomorphism \[\dvr[G] \otimes C_n'(G,\dvr) \cong C_n(G,\dvr),\] given by \(g \otimes (1,g_1,\dotsc,g_n) \mapsto (g, gg_1, \dotsc,gg_n)\), where the left hand side is equipped with the free \(\dvr[G]\)-module structure given by algebra multiplication. This exhibits \(C_n(G,\dvr)\) as a free \(\dvr[G]\)-module in the category \(\Mod_{\dvr[G]}\) of complete bornological \(\dvr[G]\)-modules. 

We now compute the ``dagger completion" of the complex \(C_\bullet(G,\dvr)\), for a finitely generated discrete group \(G\). What this specifically means is that we complete the complex \((C_\bullet(G,\dvr),\delta)\) under suitable growth conditions, which in the case \(n = 0\) specialises to the dagger completion \(\dvr[G]^\dagger\) of the group algebra. We also place conditions on functions \(G^n \to \dvr\) that link them to the underlying large geometry of the group \(G\). To this end, view \(G\) as a (discrete) proper metric space, with the word-length metric \(d(g,h) = l(g^{-1}h)\). For each \(R \in \N\), we define \(C_n(G,
\dvr)_R^\an\) as the \(\dvr\)-module of functions \(f \colon G^{n+1} \to \dvr\) of finite support that satisfy:

\begin{itemize}
\item \(f(g_0,\dotsc,g_n) = 0\) if \(d(g_i,g_j) > R\) for \(i,j \in \{0,\dotsc, n\}\).
\item For all \((g_1,\dotsc,g_n) \in G^n\), the function \(G \ni g \mapsto f(g,gg_1,\dotsc,gg_n) \in \dvr\) lies in \(\ling{\dvr[G]}\). 
\end{itemize} 

We equip this \(\dvr\)-module with the bornology where a subset \(S_{c,d} \subseteq C_n(G,\dvr)_R^\an\) of functions \(f \colon G^{n+1} \to \dvr\) as above is bounded if and only if there are \(c\), \(d>0\) such that \(\nu(f (gg_0, gg_1,\dotsc, gg_n)) +1 \geq cl(g) + d\) for all \(f \in S_{c,d}\), \(g \in G\) and \((g_0,\dotsc,g_n)\in G^{n+1}\). More symmetrically, \(f \in S_{c,d}\) if and only if \[\nu(f(g_0,\dotsc,g_n)) + 1 \geq c(l(g_0) + \cdots + l(g_n)) + d\] for all \((g_0,\dotsc,g_n) \in G^{n+1}\). Any such bounded subset is contained in the bounded submodule generated by finite sums of the form \[\sum_{(g_0,\dotsc,g_n) \in G^{n+1}} a_{(g_0,\dotsc,g_n)} \delta_{(g_0,\dotsc,g_n)}\] such that \(d(g_i,g_j) \leq R\) for all \(i,j\) and \(\nu(a_{(g_0,\dotsc,g_n)}) + 1 \geq c (l(g_0) + \cdots + l(g_n)) + d,\) for some \(c\), \(d>0\), and all \((g_0,\dotsc,g_n) \in G^{n+1}\). Viewing \(S_{c,d}\) as a subset of \(\bigoplus_{G^{n+1}} \dvr\),  its elements are finite sums of weighted tuples \(\sum_{(g_0,\dotsc,g_n) \in G^{n+1}} a_{(g_0,\dotsc,g_n)} (g_0,\dotsc,g_n)\) such that \(d(g_i,g_j) \leq R\) and \(\nu(a_{(g_0,\dotsc,g_n)}) + 1 \geq c (l(g_0) +\cdots + l(g_n)) + d\) for some \(c\), \(d>0\), and all \((g_0,\dotsc,g_n) \in G^{n+1}\). Note that by modifying the parameters \(c\) and \(d\), we may arrange that \(S_{c,d} = S_{C,1}\) for some \(C>0\), which we denote by \(S_C\) to simplify notation. 

Taking the bornological inductive limit \(C_n(G,\dvr)^\an \defeq \varinjlim_R C_n(G,\dvr)_R^\an\), we obtain a bornological \(\dvr\)-module, which is not yet complete. Taking the completion in the bornology generated by the \(S_C \subseteq C_n(G,\dvr)_R^\an\) for varying \(R\) and \(C>0\), we get a complete bornological \(\dvr\)-module, which we denote by \(C_n(G,\dvr)^\dagger\).  Finally, the differentials \(\delta_n \colon C_n(G,\dvr) \to C_{n-1}(G,\dvr)\) extend to bounded \(\dvr\)-linear maps on \(C_n(G,\dvr)^\dagger \to C_{n-1}(G,\dvr)^\dagger\). Similarly, we can define a reduced version \((\bar{C}_\bullet(G,\dvf)^\dagger, \delta)\), wherein we additionally require that the functions \(f \colon G^{n+1} \to \dvr\) in \(C_n(G,\dvr)^\dagger\) satisfy \(f(x_0,\dotsc,x_n) = 0\) if \(x_i = x_{i+1}\) for some \(i\). We equip the bornological \(\dvr\)-modules \(C_n(G,\dvr)^\dagger\) with a \(\dvr[G]^\dagger\)-module structure by extending the diagonal \(G\)-action \[g\cdot(g_1,\dotsc,g_n) \defeq (gg_1,\dotsc, gg_n),\] noting that the bornology on the completed complex \(C_n(G,\dvr)^\dagger\) is defined in a way to ensure that the diagonal action extends an action of \(\dvr[G]^\dagger\). Finally, the differentials on \(C_\bullet(G,\dvr)^\dagger\) are \(G\)-equivariant and hence extend to \(\dvr[G]^\dagger\)-module homomorphisms.

\begin{lemma}\label{lem:free-complete-dagger}
We have an isomorphism \[C_n(G,\dvr)^\dagger \cong \dvr[G]^\dagger \hot_{\dvr[G]} C_n(G,\dvr)\] of chain complexes of bornological \(\dvr[G]^\dagger\)-modules. This descends to an isomorphism \(\bar{C}_n(G,\dvr)^\dagger \cong \dvr[G]^\dagger \hot_{\dvr[G]} \bar{C}_n(G,\dvr)\) of reduced complexes. In particular, the bornological \(\dvr\)-modules \(C_n(G,\dvr)^\dagger\) and \(\bar{C}_n(G,\dvr)^\dagger\) are free as \(\dvr[G]^\dagger\)-modules. 
\end{lemma}

\begin{proof}
For a fixed \(R>0\), we have an isomorphism of bornological \(\dvr\)-modules \[\ling{\dvr[G]} \otimes C_n'(G,\dvr)_R \to C_n(G,\dvr)_R^\an, \quad  g \otimes (1,g_1,\dotsc,g_n) \mapsto (g,gg_1,\dotsc, gg_n),\] by the left invariance of the metric on \(G\). Taking the inductive limit over \(R\) and then applying the completion functor, we get an isomorphism \[\comb{\ling{\dvr[G]} \otimes C_n'(G,\dvr)} \cong \dvr[G]^\dagger \hot C_n'(G,\dvr) \cong C_n(G,\dvr)^\dagger,\] using that \(C_n'(G,\dvr)\) has the fine bornology and is therefore already complete.  To promote this to an isomorphism of \(\dvr[G]^\dagger\)-modules, we equip the left hand side with the module stucture \[\dvr[G]^\dagger \otimes \dvr[G]^\dagger \otimes C_n'(G,\dvr)_R^\dagger \to \dvr[G]^\dagger \otimes C_n'(G,\dvr), f \otimes (h \otimes k) \mapsto f \star h \otimes k,\] where \(f \star h\) is the multiplication of the algebra \(\dvr[G]^\dagger\). Now using that inductive limits commute with the completed projective bornological tensor product, we have an isomorphism of \(\dvr[G]^\dagger\)-modules \[\dvr[G]^\dagger \hot C_n'(G,\dvr) \cong C_n(G,\dvr)^\dagger,\] for each \(n \in \N\). From this it follows that \[\dvr[G]^\dagger \hot_{\dvr[G]} C_n(G,\dvr) \cong \dvr[G]^\dagger \hot_{\dvr[G]}\dvr[G] \hot C_n'(G,\dvr) \cong C_n(G,\dvr)^\dagger\] for each \(n \in \N\), as required.  The same proof applies to yield the claims about non-degenerate simplices. 
\end{proof}

We now have a chain complex \(C_n(G,\dvr)^\dagger\) of free \(\dvr[G]^\dagger\)-modules. In fact by Lemma \ref{lem:free-complete-dagger}, the chain complex \(C_\bullet(G,\dvr)^\dagger\) arises as the image of the base change functor 
\[f_! \colon \Mod_{\dvr[G]} \to \Mod_{\dvr[G]^\dagger}, \quad M \mapsto \dvr[G]^\dagger \otimes_{\dvr[G]} M,\] for the canonical map \(f \colon \dvr[G] \to \dvr[G]^\dagger\). Furthermore, since \(\dvr\) with the fine bornology is a dagger algebra,  the trivial representation on \(\dvr\) extends to a \(\dvr[G]^\dagger\)-module structure on \(\dvr\).  The augmentation map \(\alpha \colon \dvr[G] \to \dvr\) extends to a bounded \(G\)-equivariant map \(\alpha \colon \dvr[G]^\dagger \to \dvr\), so that we have a chain complex \[C_\bullet(G,\dvr)^\dagger \overset{\alpha}\to \dvr \to 0\] of free \(\dvr[G]^\dagger\)-modules. At this point it is tempting to use the identification \[\dvr[G]^\dagger \otimes_{\dvr[G]} C_n(G,\dvr) \cong C_n(G,\dvr)^\dagger\] from Lemma \ref{lem:free-complete-dagger} to declare that \(C_\bullet(G,\dvr)^\dagger\) is a free \(\dvr[G]^\dagger\)-module \emph{resolution} of \(\dvr\). But since the contracting homotopy of \(C_\bullet(G,\dvr)\) is only \(\dvr\)-linear, this fact alone does not imply that \(C_\bullet(G,\dvr)^\dagger\) is contractible. Moreover, as discussed in \cite{Meyer:Combable}*{Proposition 2}, this complex is also not necessarily contractible as only quasi-isometries induce chain homotopies on the tempered complex, and the identity is not such a map unless the group is finite. We therefore need some further information on the large scale geometry of the group:


\begin{definition}\label{def:combing}
A \textit{combing} on a proper discrete metric space \(X\) with a fixed base point \(* \in X\) is a sequence of morphisms \(f_n \colon G \to G\) such that:
\begin{itemize}
\item for all \(x \in X\), \(f_0(x) = *\) and there is an \(n \in \N\) such that \(f_N(x) = x\) for all \(N \geq n\);
\item there is a \(C > 0\) such that \(d(f_n(x), f_n(y)) \leq C(d(x,y) + 1)\) for all \(x,y \in X\) and \(n \in \N\);
\item there is an \(S > 0\) such that \(d(f_n(x), f_{n+1}(x)) \leq S\) for all \(x \in X\), \(n \in \N\). 
\end{itemize} 

For each \(x \in X\), let \(J(x)\) be the number of \(n \in \N\) such that \(f_n(x) \neq f_{n+1}(x)\). Since the sequence \((f_n(x))_{n \in \N}\) is eventually constant, \(J(x)< \infty\). We say a combing has \textit{polynomial growth} of order \(m \in \N\) if there is a \(C>0\) such that \(J(x) \leq C(l(x) + 1)^m\) for some \(m \in \N\). One can similarly combings with other growth conditions, such as (sub)-exponential growth.  
\end{definition}

The class of groups with polynomial growth combings is very large. For instance, it contains all hyperbolic groups, and therefore all finitely generated free nonabelian groups and free groups. In the case of a finitely generated free group \(\mathbb{F}_r\), consider a set of generators \(s(1),\dotsc,s(r)\) satisfying \(s(r+j) =s(j)^{-1}\), for \(j = 1,\dotsc, r\). For the word length function \(l\) relative to these generators, we define the metric \(d(g,h) = l(g^{-1} h)\). We can define a combing as follows: given a reduced word \(g \in \mathbb{F}_r\) written as \(g = s(i_1)\cdots s(i_l)\), \(f_n(g) \defeq s(i_1)\cdots s(i_n)\) for \(0 \leq n < l\), and \(f_n(g) = g\) for \(n \geq l\). It is easy to see that this really is a combing, and \(J(g) = l(g)\) for this combing, so we get a combing of linear growth.

\begin{theorem}\label{main:1}
Let \(G\) be a combable group of polynomial growth. The augmentation map \(\alpha \colon \dvr[G]^\dagger \to \dvr\) induces a chain homotopy equivalence \(\bar{C}_\bullet(G,\dvr)^\dagger \to \dvr[0]\). Consequently, \(\bar{C}_n(G,\dvr)^\dagger\) is a free \(\dvr[G]^\dagger\)-module resolution of \(\dvr\). 
\end{theorem}

\begin{proof}
Consider again the map \(H_n(f,g) \colon \bar{C}_n(G,\dvr) \to \bar{C}_{n+1}(G,\dvr)\) from Lemma \ref{lem:trivial-complex} defined by \[H_n(f,g)(x_0,\dotsc,x_n) = \sum_{j=0}^n (-1)^j (f_j^*(x_0,\dotsc,x_j), g_{n-j}^*(x_j,\dotsc,x_n)).\] Since the codomain is non-degenerate simplicies, \(H_n(f,g)(x_0,\dotsc,x_n)\) is only nonzero if \(f(x_j) \neq g(x_j)\) for some \(j\). Since \((f_n)\) is a combing on \(G\), for each \(x_i\), there is a \(j\) such that \(f_n(x_i) = f_{n+1}(x_i)\) for all \(n \geq j\). Using this, we see that on a basis vector \((x_0,\dotsc,x_n) \in C_n(G,\dvr)\), the sum \[H_n(x_0,\dotsc,x_n) \defeq \sum_{j=0}^\infty H_n(f_j,f_{j+1})(x_0,\dotsc,x_n)\] has \(J(x_0)+J(x_1)+\cdots + J(x_n)\) nonzero terms. Furthermore, we have \[\delta \circ H + H \circ \delta = \mathrm{id},\] so that \(H\) is a contracting homotopy for \(\bar{C}_\bullet(G,\dvr)\). 

We need to show that \(H\) extends to a contracting homotopy for \(\bar{C}_\bullet(G,\dvr)^\dagger\). By the universal property of completions, it suffices to show that \(H_n \colon C_n(G,\dvr)^\an \to C_{n+1}(G,\dvr)^\an\) is bounded. The support condition is preserved by \(H_n\) since it is preserved by each \(H_n(f_j,f_{j+1})\), which in turn follows from the uniform closedness of the combing. It remains to check the control condition. Let \[\sum_{(g_0,\dotsc,g_n) \in G^{n+1}} a_{(g_0,\dotsc,g_n)} (g_0,\dotsc,g_n) \in S_c\] be an element in a bounded subset of \(C_n(G,\dvr)_R^\an\), for some \(R>0\) and \(c>0\). Note that only finitely many coefficients in each element of such a sum in \(S_c\) is nonzero. Applying \(H_n\) to such a sum, we get \begin{multline*}
\sum_{(g_0,\dotsc,g_n) \in G^{n+1}} a_{(g_0,\dotsc,g_n)} H_n(g_0,\dotsc,g_n) \\
= \sum_{(g_0,\dotsc, g_n) \in G^{n+1}} a_{(g_0,\dotsc,g_n)} \sum_{i=0}^{J(g_0) + \cdots + J(g_n)} (-1)^i (f_j(g_0),\dotsc,f_j(g_i),f_{j+1}(g_i),\dotsc, f_{j+1}(g_n)) \\
= \sum_{i=0}^{J(g_0) + \cdots + J(g_n)} (-1)^i \sum_{(g_0,\dotsc, g_n) \in G^{n+1}} a_{(g_0,\dotsc,g_n)}(f_j(g_0),\dotsc,f_j(g_i),f_{j+1}(g_i),\dotsc, f_{j+1}(g_n)). 
\end{multline*} It remains to show that for each \(i\), the sum \[(-1)^i \sum_{(g_0,\dotsc, g_n) \in G^{n+1}} a_{(g_0,\dotsc,g_n)}  (f_j(g_0),\dotsc,f_j(g_i),f_{j+1}(g_i),\dotsc, f_{j+1}(g_n))\] has the right growth property. To this end, we note that \(l(f_j(g)) + 1 \leq D (l(g) +1)\) for some \(D>0\) and all \(j \in \N\) and \(g \in G\), which follows from the uniform quasi-Lipschitz condition and the finiteness of the set \(f_j(e)\). This ensures that \[\nu(a_{(g_0,\dotsc,g_n)}) + 1 \geq c (l(g_0) + l(g_1) + \cdots + l(g_n) + 1) \geq  \frac{c}{D}(l(f_j(g_0)) + \cdots + l(f_{j+1}(g_n)) + 1),\] showing that \(H_n\) of each summand of its image on \(S_c\) lies in \(S_{\frac{c}{D}}\). As there are \(J(x_0) + \cdots + J(x_n)\) such summands, the polynomial growth condition now proves the desired result. 
\end{proof}


So far we have shown that for a combable group \(G\) and the free resolution \(\bar{C}_\bullet(G,\dvr) \to \dvr\) of \(\dvr\) by \(\dvr[G]\)-modules, the kernel \(\tilde{\bar{C}}_\bullet(G,\dvr)^\dagger\) of the induced chain complex \(\dvr[G]^\dagger \hot_{\dvr[G]} \bar{C}_\bullet(G,\dvr) \cong \bar{C}_{\bullet}(G,\dvr)^\dagger \to \dvr\) of \(\dvr[G]^\dagger\)-modules is contractible. In other words, \(\bar{C}_\bullet(G,\dvr)^\dagger \to \dvr\) is a resolution of \(\dvr\) by free \(\dvr[G]^\dagger\)-modules. We now use this resolution to arrive at a resolution of \(\dvr[G]^\dagger\) by free \(\dvr[G]^\dagger\)-modules. Note that by tensoring with \(\dvr[G]^\dagger\), we get a resolution \[\bar{C}_\bullet(G,\dvr)^\dagger \hot \dvr[G]^\dagger \to \dvr \otimes \dvr[G]^\dagger \cong \dvr[G]^\dagger\] by free \(\dvr[G]^\dagger\)-bimodules, but the induced  \(\dvr[G]^\dagger\)-bimodule structure on \(\dvr[G]^\dagger\) given by \(f_1 \cdot f_2 \cdot f_3 = f_1(e)f_2f_3\) for \(f_1\), \(f_2 \in \dvr[G]^\dagger\), \(v \in \dvr\) is not the natural one obtained from algebra multiplication. Following \cite{meyer2004embeddings}, we call a convolution algebra such as \(\dvr[G]^\dagger\) is \textit{symmetric} if the operators \(U(\phi)(g,h) = \phi(gh, g)\) and \(U^{-1}\phi(g,h) = \phi(h,h^{-1}g)\) are bounded on \(\dvr[G]^\dagger \hot \dvr[G]^\dagger \cong \dvr[G \times G]^\dagger\). 

\begin{lemma}\label{lem:symmetric-convolution}
The bornological algebra \(\dvr[G]^\dagger\) is symmetric.
\end{lemma}

\begin{proof}
This follows from the estimate that \[\nu(U(\phi(g,h))) + 1 = \nu(\phi(gh,g)) + 1 \geq \frac{l(gh,g)}{m} = \frac{l(gh) + l(g)}{m} \geq \frac{\max \{l(g),l(h)\}}{m}\] for some \(m \geq 1\). Now consider \(d_1((g,h),(g',h')) = l_1(g^{-1}g',h^{-1}h')\), where \(l_1(g,h) = \max \{l(g), l(h)\}\). Setting \((g,h) = (e,e)\), we get \(l_1(g,h) = \max \{l(g), l(h)\}\), which is another length function. The metric \(d_2\) on \(G \times G\) induced by the length function \(l(g,h) = l(g) + l(h)\) is equivalent to \(d_1\), so there is a \(C>0\) such that \(l_1(g,h) = d_1((e,e),(g,h)) \geq C d_2((e,e),(g,h)) = C( l(g) + l(h))\). This shows that \(U\) is bounded. In a similar manner, one can see that \(U^{-1}\) is bounded. 
\end{proof}

\begin{proposition}\label{prop:isocohomological-criterion}
Let \(G\) be a finitely generated discrete group. The canonical map \(\dvr[G] \to \dvr[G]^\dagger\) is an isocohomological embedding if and only if \[\dvr[G]^\dagger \hot_{\dvr[G]}^{\mathbb{L}} \dvr \cong \dvr[G]^\dagger \otimes_{\dvr[G]^\dagger}^{\mathbb{L}} \dvr \cong \dvr.\]
\end{proposition}
\begin{proof}
The proof of \cite{meyer2004embeddings}*{Proposition 41} works after making appropriate modifications, using that \(\dvr[G]^\dagger\) is a symmetric convolution algebra by Lemma \ref{lem:symmetric-convolution}. For clarity, we provide the proof. Given any complete bornological modules \(M \in \Mod_{\dvr[G]^\op}\) and \(N \in \Mod_{\dvr[G]}\), equip their completed tensor product \(M \hot N\) with the inner conjugation action \((m \otimes n) \cdot g \defeq (mg \otimes g^{-1} n)\). For the associated right \(\dvr[G]\)-module structure on \(M \hot N\), we then have a bornological isomorphism \(M \hot_{\dvr[G]} N \cong (M \hot N) \hot_{\dvr[G]} \dvr\). This extends to an isomorphism at the level of chain complexes. And, if \(M\) or \(N\) is projective, so is \(M \hot_{\dvr[G]} N\). This implies that for \(M \in \mathsf{D}(\dvr[G]^\op)\) and \(N \in \mathsf{D}(\dvr[G])\), we have \(M \hot_{\dvr[G]}^{\mathbb{L}} N \cong (M \hot N) \hot_{\dvr[G]}^{\mathbb{L}} \dvr\). Applying this in particular to \(M = N = \dvr[G]^\dagger\), we get \[\dvr[G]^\dagger \hot_{\dvr[G]}^{\mathbb{L}} \dvr[G]^\dagger \cong (\dvr[G]^\dagger \hot \dvr[G]^\dagger) \otimes_{\dvr[G]}^{\mathbb{L}} \dvr,\] where \(\dvr[G]^\dagger \hot \dvr[G]\) has the inner conjugation action. But by definition, the bounded (invertible) operator \(U \colon \dvr[G]^\dagger \hot \dvr[G]^\dagger \to \dvr[G]^\dagger \hot \dvr[G]^\dagger\) intertwines the inner conjugation action with the \(G\)-action on \(\dvr[G]^\dagger \hot \dvr[G]^\dagger\) defined by \( \phi \cdot g(x,y) \defeq \phi(x,yg)\) for \(\phi \in \dvr[G \times G]^\dagger\), and \(x\), \(y\), \(g \in G\). Identifying \(\dvr[G]^\dagger \hot \dvr[G]^\dagger\) with this latter \(G\)-action, we obtain an isomorphism \[\dvr[G]^\dagger \hot_{\dvr[G]}^{\mathbb{L}} \dvr[G]^\dagger \cong \dvr[G]^\dagger \hot \dvr[G]^\dagger \hot_{\dvr[G]}^{\mathbb{L}} \dvr,\] which yields the desired result. \qedhere
\end{proof}



\begin{theorem}\label{theorem:main1}
Let \(G\) be a combable group of polynomial growth. Then the canonical map \(\dvr[G] \to \dvr[G]^\dagger\) is an isocohomological embedding.
\end{theorem}

\begin{proof}
By Proposition \ref{prop:isocohomological-criterion}, we need to show that \(\dvr[G]^\dagger \hot_{\dvr[G]}^{\mathbb{L}} \dvr \simeq \dvr[0]\). The derived tensor product on the left hand side is computed by \[\dvr[G]^\dagger \hot_{\dvr[G]}^{\mathbb{L}} \dvr \cong \dvr[G]^\dagger \hot_{\dvr[G]} P_\bullet,\] where \(P_\bullet \to \dvr\) is a projective \(\dvr[G]\)-module resolution of \(\dvr\). So the map \(\dvr[G] \to \dvr[G]^\dagger\) is isocohomological if and only if \(\dvr[G]^\dagger \hot_{\dvr[G]}^{\mathbb{L}} P_\bullet\) is contractible, with \(P_\bullet\) as above. Choosing \(P_\bullet\) as \(\tilde{\bar{C}}_\bullet(G,\dvr) = \ker(\alpha \colon \bar{C}_\bullet(G,\dvr) \to \dvr)\), we get such a resolution. And by Theorem \ref{main:1}, \(\dvr[G]^\dagger \hot_{\dvr[G]} \tilde{\bar{C}}_\bullet(G,\dvr) \cong \tilde{\bar{C}}_\bullet(G,\dvr)^\dagger\) is contractible, which completes the proof. 
\end{proof}

Next, we look at the nonarchimedean version of the \textit{noncommutative torus}, which is a special case of a twisted monoid algebra defined in \cite{Meyer-Mukherjee:Bornological_tf}*{Section 6}. Actually, we prove a more general result for certain types of crossed product algebras, following \cite{Meyer-Mukherjee:Bornological_tf}*{Section 7}. Let \(B\) be a unital, complete bornological \(\dvr\)-algebra, and \(G\) a finitely generated discrete group acting on \(B\) by bounded algebra homomorphisms \(\alpha \colon G \to \mathsf{End}(B)\).  We define the \textit{crossed product} \(G \rtimes_{\alpha} B\) of \(B\) by the action \(\alpha\) as the \(\dvr\)-module \(\dvr[G] \hot B \cong \dvr[G,B]\) with the tensor product bornology induced by the bornology on \(B\) and the fine bornology on \(\dvr[G]\). Equipped with the multiplication \[(\sum_{g \in G} a_g \delta_g) \cdot (\sum_{h \in G} b_h \delta_h) = \sum_{g,h \in G} a_g \alpha_g(b_h) \delta_{gh},\] we get a unital bornological \(\dvr\)-algebra. 

\begin{theorem}\label{theorem:crossed-product}
Let \(\alpha \colon G \to \mathsf{End}(B)\) be an action of a finitely generated discrete group on a unital, complete bornological \(\dvr\)-algebra. Suppose further that the canonical maps \(B \to B^\dagger\) and \(\dvr[G] \to \dvr[G]^\dagger\) are isocohomological embeddings. Then the canonical map \(G \rtimes_{\alpha} B \to (G \rtimes_{\alpha^\dagger} B^\dagger)^\dagger\) is an isocohomological embedding. 
\end{theorem}

\begin{proof}
It is easy to see that the tensor product of isocohomological embeddings is an isocohomological embedding. Consequently, it suffices to prove that \(G \rtimes_{\alpha} B \to G \rtimes_{\alpha^\dagger} B^\dagger\) and \(G \rtimes_{\alpha^\dagger} B^\dagger \to (G \rtimes_{\alpha^\dagger} B^\dagger)^\dagger\) are isocohomological embeddings. Consider the algebra homomorphism \(B \to G \rtimes_{\alpha} B\), and the induced functor \(i_{!}^B \colon \mathsf{Mod}_B \to \mathsf{Mod}_{G \rtimes_\alpha B}\), \(M \mapsto (G \rtimes_\alpha B) \hot_B M\). Since \(G \rtimes_\alpha B = \dvr[G] \hot B\) is projective as a right \(B\)-module, the functor \(i_{!}^B\) is exact. Consequently, we have \(G \rtimes_{\alpha^\dagger} B^\dagger \cong (G \rtimes_\alpha B) \hot_B B^\dagger \cong (G \rtimes_\alpha B) \hot_{B}^{\mathbb{L}} B^\dagger\). Now consider 
\begin{multline}\label{isocohomological-crossed}
(G \rtimes_{\alpha^\dagger} B^\dagger) \hot_{G \rtimes_{\alpha} B}^{\mathbb{L}} G \rtimes_{\alpha^\dagger} B^\dagger \cong G \rtimes_{\alpha^\dagger} B^\dagger \hot_{G \rtimes_{\alpha} B}^{\mathbb{L}} (G \rtimes B) \hot_{B}^{\mathbb{L}} B^\dagger \\
G \rtimes_{\alpha^\dagger} B^\dagger \hot_{B}^{\mathbb{L}} B^\dagger \cong G \rtimes_{\alpha^\dagger} (B^\dagger \hot_{B}^{\mathbb{L}} B^\dagger \cong G \rtimes_{\alpha^\dagger} B^\dagger,
\end{multline} which shows that \(G \rtimes_\alpha B \to G \rtimes_{\alpha^\dagger} B^\dagger\) is isocohomological.  

Now consider the map \(\dvr[G] \to G \rtimes_{\alpha^\dagger} B^\dagger\), which induces the functor \(i_{!}^G \colon \mathsf{Mod}_{\dvr[G]} \to \Mod_{G \rtimes_{\alpha^\dagger} B^\dagger}\), defined by \(M \mapsto G \rtimes_{\alpha^\dagger} B^\dagger \hot_{B^\dagger} M\). Since \(G \rtimes_{\alpha^\dagger} B^\dagger = \dvr[G] \hot B^\dagger\) is projective as a right \(B^\dagger\)-module, the functor \(i_{!}^G\) is exact. It is easy to see that \[i_{!}^G(\dvr[G]^\dagger) \cong \dvr[G]^\dagger \hot B^\dagger \cong (G \rtimes_{\alpha^\dagger} B^\dagger)^\dagger,\] which with a similar argument as in Equation \ref{isocohomological-crossed} shows that \(G \rtimes_\alpha B^\dagger \to (G \rtimes_{\alpha^\dagger} B^\dagger)^\dagger\) is isocohomological, provided \(\dvr[G] \to \dvr[G]^\dagger\) is one.    
\end{proof}

We now apply Theorem \ref{theorem:crossed-product} to the nonarchimedean noncommutative torus, defined in \cite{Meyer-Mukherjee:Bornological_tf} as follows: let \(\dvr^* = \setgiven{x \in \dvr}{\abs{x} = 1}\), and consider the normalised \(2\)-cocycle \(c \colon \Z^2 \times \Z^2 \to \dvr\) defined by \(c((m_1,n_1),(m_2,n_2)) = \lambda^{n_1 m_2}\), where \(\lambda \in \dvr^*\) is some fixed scalar. We equip the \(\dvr\)-module \(\dvr[\Z^2]\) with the twisted convolution defined by 
\begin{multline*}
\sum_{(m,n) \in \Z^2}x_{(m,n)} \delta_{(m,n)} * \sum_{(m',n') \in \Z^2} y_{(m',n')} \delta_{(m',n')} = \\
 \sum_{(m,n),(m',n') \in \Z^2} x_{(m,n)}y_{(m',n')}c((m,n),(m',n')) \delta_{(m+m',n+n')}.
\end{multline*}

Denote \(\dvr[\Z^2]\) with the twisted convolution by \(\dvr[\Z^2,c]\). The twist in the multiplication does not change the linear growth bornology, so the completion \(\dvr[G,c]^\dagger\) is still given by the same growth conditions as in the definition of the dagger completion of a group algebra.  We call this complete (unital) bornological algebra the \(\dvgen\)-adic noncommutative torus. To justify this, consider \(U_1 \defeq \delta_{(1,0)}\) and \(U_2 = \delta_{(0,1)}\); these elements generate \(\dvr[\Z^2,c]\) as a \(\dvr\)-algebra, and satisfy \[U_2 U_1 = \lambda U_1 U_2.\] The dagger completion \(\dvr[\Z^2,c]^\dagger\) is isomorphic as a \(\dvr\)-module to the dagger completion of the Laurent polynomial algebra \(\dvr[U_1^{\pm 1}, U_2^{\pm 1}]\) in the variables \(U_1\) and \(U_2\) satisfying \(U_2 U_1 = \lambda U_1 U_2\). 

In order to use Theorem \ref{theorem:crossed-product} in the most convenient manner, we first describe the noncommutative torus as a crossed product algebra. Consider the action \(\alpha\) of \(\Z\) on \(\dvr[t,t^{-1}]\) by \(n \cdot t^m \defeq \lambda^n t^m\) for some \(\lambda \in \dvr^*\). Since \(\lambda\) is a unit, the action is \textit{uniformly bounded} in the sense of \cite{Meyer-Mukherjee:Bornological_tf}*{Definition 7.1}. By \cite{Meyer-Mukherjee:Bornological_tf}*{Proposition 7.2}, this action extends to a uniformly bounded action \(\alpha^\dagger\) on the dagger completion \(\dvr[t,t^{-1}]^\dagger\). We then have the following:

\begin{lemma}\label{lem:equivalent-nc-torus}
We have isomorphisms \((\Z \rtimes_{\alpha} \dvr[t,t^{-1}])^\dagger \cong (\Z \rtimes_{\alpha^\dagger} \dvr[t,t^{-1}]^\dagger)^\dagger \cong \dvr[\Z^2,c]^\dagger\) of bornological algebras. 
\end{lemma}

\begin{proof}
The first isomorphism follows from the universal property of dagger completions. To show that \((\Z \rtimes_\alpha \dvr[t,t^{-1}])^\dagger \cong \dvr[\Z^2,c]^\dagger\), it suffices to show that \(\Z \rtimes_\alpha \dvr[t,t^{-1}]\) and \(\dvr[\Z^2,c]\) are isomorphic as bornological \(\dvr\)-algebras, since the multiplications are defined by extension to the completion. The latter claim is purely algebraic and is well-known. 
\end{proof}

\begin{theorem}\label{thm:noncommutative-torus-isocohomological}
The canonical map \(\dvr[\Z^2,c] \to \dvr[\Z^2,c]^\dagger\) is an isocohomological embedding. 
\end{theorem}

\begin{proof}
By Lemma \ref{lem:equivalent-nc-torus}, we need to show that canonical map \(\Z \rtimes_\alpha \dvr[t,t^{-1}] \to (\Z \rtimes_\alpha \dvr[t,t^{-1}])^\dagger\) is isocohomological. Equpping \(\Z\) with its natural length combing, Theorem \ref{main:1} shows that \(\dvr[\Z] \to \dvr[\Z]^\dagger\) is isocohomological. Consequently, Theorem \ref{theorem:crossed-product} applies to yield the desired result.  
\end{proof}

\subsection{The case of reductive \(p\)-adic groups}\label{subsection:reductive}

In this section, let \(G\) be a reductive \(p\)-adic group, and let \(H\) be a maximal compact open subgroup. This implies that the quotient space \(X \defeq G/H\) is discrete. It turns out that we can construct a projective bimodule \(\dvr[G]^\dagger\)-bimodule resolution of \(\dvr\) in a manner similar to the word-hyperbolic case. There is however a slight technical subtlety arising from the fact that the explicit description of the dagger completion in \cite{Meyer-Mukherjee:Bornological_tf}*{Section 6} only covers finitely generated discrete groups. The general case expresses a group as a colimit \(G = \varinjlim_i G_i\) of finitely generated groups \(G_i\), and the corresponding dagger completion as a colimit \(\dvr[G]^\dagger \cong \varinjlim_i \dvr[G_i]^\dagger\) taken in the category of dagger algebras. We prefer, however, to work with a different dagger algebra that is constructed using a natural length function on a reductive \(p\)-group. Indeed, let \(G\) be a compactly generated, locally compact topological group. Then there is a compact subset \(S\) with \(S = S^{-1}\) such that \(G = \bigcup_{n=0}^\infty S^n\). For simplicity, assume that the identity lies in \(S\). Define \(l \colon G \to \N\) as the minimal \(n \in \N\) such that \(g \in S^n\). Since \(G\) is totally disconnected, it is convenient to assume that \(S\) is \(H\)-biinvariant, which ensures that the length function on \(G\) is \(H\)-biinvariant. Consequently, the length function on \(G\) restricts to a standard word-length function on the space of double-cosets \(G//H\), which we will use to describe the bornology on the completion of the bar complex of \(X = G/H\). 

We now describe the ``dagger completion" of \(\dvr[G]\) as a bornological algebra. First, we equip the algebra \(\dvr[G]\) of finitely supported functions on \(G\) with the bornology generated by the submodules \(\dvr[F_n]\), where \(F_n \defeq \setgiven{g \in G}{l(g) \leq n}\). Denote \(\dvr[G]\) with this bornology by \(\dvr[G]_K\). Note that the fine bornology does not work here as \(F_1 = S\) is an infinite set, so restricting to finitely generated submodules does not yield a bornological algebra that takes the topology of \(G\) into account in any meaningful way. There is, however, a canonical bounded \(\dvr\)-algebra homomorphism \(\dvr[G] \to \dvr[G]_K\), since any finitely generated \(\dvr\)-submodule of \(\dvr[G]\) is contained in \(\dvr[F_n]\) for some \(n\). 

\begin{lemma}\label{lem:new-borno-tf}
The bornological \(\dvr\)-module \(\dvr[G]_K\) is bornologically torsionfree. 
\end{lemma} 
\begin{proof}
Suppose \(\dvgen U \subseteq \dvr[G]_K\) is bounded for some subset \(U \subseteq \dvr[G]\), then there is an \(n \in \N\) such that \(\dvgen U \subseteq \dvr[F_n]\). Now \(\dvgen^{-1}\dvr[F_n] = \setgiven{\delta_g \in \dvr[G]}{\dvgen \delta_g \in \dvr[F_n]}\) is bounded as it can be identified with a submodule of \(\dvr[F_n]\), which follows from the fact that multiplication by \(\dvgen\) on \(\dvr[G]\) is injective. The claim now follows from the observation that \(U \subseteq \dvgen^{-1}\dvr[F_n]\).
\end{proof} 

The linear growth bornology on this bornological algebra can now be defined as before as the bornology generated by submodules of the form \[M_n = \sum_{j=0}^\infty \dvgen^j \dvr[F_n]^{j+1},\] so that \(\ling{\dvr[G]} = \varinjlim_{n \in \N} M_n\). Completing in this bornology yields \[\dvr[G]_K^\dagger = \setgiven{\sum_{g \in G} x_g \delta_g}{\nu(x_g) + 1 \geq c\cdot l(g), c>0},\] with the bornology where a subset \(S_c\) of such power series is bounded if and only if the coefficients satisfy the growth condition above for the same constant \(c>0\). The remainder of this section will be dedicated to showing the following:

\begin{theorem}\label{main:isoco-p-adic}
The canonical map \(\dvr[G] \to \dvr[G]_K^\dagger\) is an isocohomological embedding.
\end{theorem}

To prove Theorem \ref{main:isoco-p-adic}, we use the same strategy as before, constructing a resolution of \(\dvr\) by \(\dvr[G]\)-modules, and using the extending the resulting contracting homotopy to a suitable completion of the bar complex.  Let \(X_n \defeq X^{n+1}\) and \(C_n(X,\dvr)\) be the \(\dvr\)-module spanned by finitely supported functions \(f \colon X_n \to \dvr\) on \(X_n\). Then \(C_n(X,\dvr)\) with the same differential \(\delta_n \colon C_n(X,\dvr) \to C_{n-1}(X,\dvr)\) as for the bar complex on \(X\) in Section \ref{subsec:resolutions-trivial}, with diagonal action of \(G\) yields a projective \(\dvr[G]\)-module resolution of \(\dvr\), where the latter is equipped with the trivial \(G\)-action. Likewise, we can define a reduced bar complex \(\bar{C}_\bullet(X,\dvr)\) on \(X\), wherein the functions \(f \colon X_n \to \dvr\) vanish on degenerate simplices. To construct a resolution of \(\dvr\) by projective \(\dvr[G]_K^\dagger\)-modules, we need some modifications from the discrete case to define the support condition.

\begin{definition}\label{def:control-support}
Given a finite subset \(F \subseteq X\), we define a relation \( \sim_F\) on \(X\) by \[x \sim_F y \text{ if and only if } x^{-1} y \in HFH,\] where \(x^{-1} y\) is viewed as an element of the double coset space \(G//H\). A subset \(S \subseteq X_n\) is \textit{controlled by a finite set \(F\)} if \(x_i \sim_F x_j\) for all \((x_0,\dotsc,x_n) \in S\) and all \(i,j \in \{0,\dotsc, n\}\). We call \(S \subseteq X_n\) \textit{controlled} if it is controlled by some finite set.  
\end{definition}

An equivalent characterisation of Definition \ref{def:control-support} is that a subset \(S \subseteq X_n\) is controlled if and only if there is a finite subset \(F \subseteq X\) such that \(S \subseteq G \cdot F\).

Now consider the \(\dvr\)-module \(C_n(X,\dvr)_F^\an\) spanned by functions \(f\colon X_n \to \dvr\) whose supports are controlled by a fixed finite set \(F\), satisfying \(\nu(f(gx_0,\dotsc,gx_n)) + 1 \geq cl(g)\) for all \(g \in G\) and each \((x_0,\dotsc,x_n) \in X_n\), and some \(c>0\). We equip \(C_n(X,\dvr)_F^\an\) with the bornology where a subset \(T_c \subseteq C_n(X,\dvr)_F^\an\) is bounded if and only if its constituent functions \(f \colon X_n \to \dvr\) satisfy the growth condition \(\nu(f(gx_0,\dotsc,gx_n)) + 1 \geq cl(g)\) for the same \(c>0\). Taking the inductive limit, we get a bornological \(\dvr\)-module \(C_n(X,\dvr)^\an = \varinjlim_{F \subseteq X} C_n(X,\dvr)_F^\an\). Its completion in the inductive limit bornology yields a complete bornological \(\dvr\)-module, which we denote by \(C_n(X,\dvr)^\dagger\) for each \(n \in \N\).

\begin{lemma}\label{lem:base-change-free}
We have an isomorphism \(\dvr[G]_K^\dagger \hot_{\dvr[G]} C_\bullet(X,\dvr) \cong C_\bullet(X,\dvr)^\dagger\) of chain complexes of bornological \(\dvr[G]_K^\dagger\)-modules. 
\end{lemma}

\begin{proof}
For a finite subset \(F \subseteq X\), let \(C_n'(X,\dvr)_F\) denote the finitely generated submodule spanned by \(([e], x_1, \dotsc, x_n)\), such that \(x_i \sim_F x_j\) and \([e] \sim_F x_i\) for all \(i\), \(j \in \{1,\dotsc, n\}\). We then have an isomorphism of \(\dvr\)-modules given by \[\dvr[G] \otimes C_n'(X,\dvr)_F \cong C_n(X,\dvr)_F, \quad g \otimes ([e],x_1,\dotsc,x_n) \mapsto ([g], gx_1,\dotsc, gx_n),\] using the \(G\)-invariance of the relations on \(X\). Equipping \(\dvr[G]\) with the linear growth bornology relative to the bornology generated by the \(\dvr[F_n]\), we get an induced isomorphism of bornological \(\dvr\)-modules \[\ling{\dvr[G]} \otimes C_n'(X,\dvr)_F^\an \to C_n(X,\dvr)_F^\an.\]  Now taking inductive limits over the finite subsets \(F\), we get an isomorphism  \[\ling{\dvr[G]} \otimes C_n'(X,\dvr) \cong C_n(X,\dvr)^\an\] of bornological \(\dvr\)-modules. Equipping the left hand side with the standard module \(\ling{\dvr[G]}\)-structure given by \(f \cdot (h \otimes k) \defeq (f \star h) \otimes k\), we promote the above to a module homomorphism, which then extends to the completion. This exhibits \(\dvr[G]^\dagger \hot C_\bullet(X,\dvr)\) as a free \(\dvr[G]^\dagger\)-module. It now follows as in the discrete case that \(\dvr[G]^\dagger \hot_{\dvr[G]} C_\bullet(X,\dvr) \cong C_\bullet(X,\dvr)^\dagger\), which also yields an (on-the-nose) isomorphism of chain complexes. 
\end{proof}

Now since \(\dvr\) is a dagger algebra with the fine bornology, the trivial representation of \(G\) on \(\dvr\) extends to a \(\dvr[G]_K^\dagger\)-module structure on \(\dvr\). Furthermore, similar to the proof of \ref{lem:symmetric-convolution}, one can check that the algebra \(\dvr[G]_K^\dagger\) is symmetric. The augmentation map \[\alpha \colon C_0(X,\dvr) \to \dvr, \quad \phi \mapsto \sum_{x \in X} \phi(x)\] yields a chain complex \(C_\bullet(G,\dvr)^\dagger \to \dvr\) of \(\dvr[G]^\dagger\)-modules. To conclude that it yields a resolution, we require the following notion of a combing of polynomial growth:

\begin{definition}\label{def:polynomial-growth-reductive}
A sequence \((p_k)\) of functions on \(X\) is called a \textit{combing} of \(X\) if it has the following properties:

\begin{itemize}
\item there is a finite subset \(F \subseteq X\) such that \(p_k(x) \sim_F p_{k+1}(x)\) for all \(k \in \N\), \(x \in X\);
\item for any finite subset \(F \subseteq X\), there is a finite subset \(\tilde{F}\subseteq X\) such that \(p_k(x) \sim_{\tilde{F}} p_k(y)\) for all \(k \in \N\) and \(x,y \in X\) with \(x \sim_F y\). 
\end{itemize}

We say that a combing has \textit{polynomial growth} with respect to a scale \(\sigma\) if the least \(k_0\) such that \(p_k(gH) = gH\) for all \(k \geq k_0\) grows at most polynomially in \(\sigma(g)\). 
\end{definition}

The following is well-known, and we only provide enough detail to keep the article self-contained. 

\begin{proposition}\label{prop:combing-p-adic}
Let \(G\) be a reductive \(p\)-adic group. Then \(X \defeq G/H\) has a combing of polynomial growth. 
\end{proposition}

\begin{proof}
Let \(\mathcal{BT}\) be the affine Bruhat-Tits building on \(G\); the action of \(G\) on \(\mathcal{BT}\) is isometric, proper and cocompact. Consider \(\xi_0 \in \mathcal{BT}\), \(H \defeq \mathsf{Stab}(\xi_0)\) and \(X = G/H\). Then \(H\) is a compact open subgroup, and we identify \(X\) with the discrete subset \(G \xi_0 \subseteq \mathcal{BT}\). To construct a combing on \(X\), we first construct a combing on \(\mathcal{BT}\). For \(\xi \in  \mathcal{BT}\), consider the unit speed geodesic map \[p(\xi) \colon [0,d(\xi,\xi_0)] \to \mathcal{BT}, \quad t \mapsto p_t(\xi),\] which is extended to all of \(\mathcal{BT}\) by setting \(p_t(\xi) \defeq \xi\) for all \(t> d(\xi, \xi_0)\). Such a function exists as \(\mathcal{BT}\) is a metric space of non-positive curvature. Varying the \(t\) over \(\N\), we obtain functions \(p_k \colon \mathcal{BT} \to \mathcal{BT}\). These have linear growth by \cite{meyer2006homological}*{Lemma 23}. Now using that \(G\) acts cocompactly on \(\mathcal{BT}\), there is an \(R>0\) such that for any \(\xi \in \mathcal{BT}\), there is a \(\xi' \in G \xi_0\) such that \(d(\xi,\xi')< R\). Let \(p_k'(\xi)\) for \(\xi \in \mathcal{BT}\) be a point in \(G \xi_0\) such that \(d(p_k(\xi), p_k(\xi'))<R\).  The restriction of these maps to \(G \xi_0 \subseteq \mathcal{BT}\) yields the required combing on \(X\).       
\end{proof}

With this choice of \(H\) and \(X\) in the definitions of the bar complex, we now have the analogue of Theorem \ref{main:1} for reductive \(p\)-adic groups:

\begin{theorem}\label{thm:main:2}
Let \(G\) be a reductive \(p\)-adic group. Then the subcomplex \(\tilde{\overline{C}}_\bullet(X,\dvr)^\dagger\) is contractible, so that \(\overline{C}_\bullet(X, \dvr)^\dagger\) is a resolution of \(\dvr\) by \(\dvr[G]^\dagger\)-modules. 
\end{theorem}

\begin{proof}
Let \((p_k) \colon X \to X\) be the combing of polynomial growth of Proposition \ref{prop:combing-p-adic}. As in the proof of Theorem \ref{theorem:main1}, we use the contracting homotopy \[H_n \colon \overline{C}_n(X,\dvr) \to \overline{C}_{n+1}(X,\dvr), \quad (x_0,\dotsc,x_n) \mapsto \sum_{i=0}^\infty H_n(p_i,p_{i+1})(x_0,\dotsc,x_n)\] for the kernel of the augmentation map on bar complex on \(X\). Notice that for each \(n\), \(H_n(p_i,p_{i+1})(x_0,\dotsc,x_n)\) is nonzero if and only if \(p_k(x_i) \neq p_{k+1}(x_j)\), and since \((p_k)\) is eventually constant, there are only finitely many \(k\)'s for which this can happen. Therefore, on each basis vector, at most finitely many summands of \(H_n\) are nonzero. To see that it preserves functions of controlled support, let \(S \subseteq X_n\) be controlled by a finite subset \(F \subseteq X\). Since \((p_k)\) is a combing, there is a finite subset \(\tilde{F} \subseteq X\) such that \(p_k(x_i) \sim_{\tilde{F}} p_{k+1}(x_i)\) and \(p_k(x_i) \sim_{\tilde{F}} p_k(x_j)\) for all \(k \in \N\) and \(x_i \in S\). Then setting \(F' \defeq H\tilde{F}H \tilde{F} H \subseteq X\), we see that \(x \sim_{\tilde{F}} y \sim_{\tilde{F}} z\) implies \(x \sim_{F'} z\). Therefore \((p_i(x_0), \dotsc, p_i(x_j),p_{i+1}(x_j), \dotsc, p_{i+1}(x_n))\) is controlled by \(F'\) for all \(x_i \in S\). Consequently, every summand of \(H_n(x_0,\dotsc,x_n)\) is controlled by \(F'\), for \(x_i \in S\). 

Finally, to verify that \(H_n\) also preserves the growth condition in the bornology on \(\overline{C}_\bullet(G,\dvr)^\an\), let \(S \subseteq \overline{C}_\bullet(G,\dvr)^\an\) be a bounded submodule. Concretely, such a bounded submodule is generated by linear combinations of elements of the form \(\sum_{(x_0,\dotsc,x_n) \in X_n} a_{(x_0,\dotsc,x_n)} (x_0,\dotsc, x_n)\), with \(x_i \sim_F x_j\) for all \(x_i \in S\) for some finite subset \(F \subseteq X\), and \(\nu(a_{(x_0,\dotsc,x_n)}) + 1 \geq c(l(x_0) + l(x_1) + \cdots + l(x_n) + 1)\), where \(l\) is the restriction of the scale on \(G\) to \(X\). The computation that applying \(H_n\) preserves the growth condition, upon enlarging the finite subset to \(F'\), now goes through exactly as in the proof of Theorem \ref{main:1}. 
\end{proof}

\begin{proof}[Proof of Theorem \ref{main:isoco-p-adic}]
Proposition \ref{prop:isocohomological-criterion} applies to any symmetric convolution algebra, in particular, the algebra \(\dvr[G]_K^\dagger\). Combine this with Theorem \ref{thm:main:2}.
\end{proof}

\section{Applications to cyclic homology}\label{section:applications-HC}

In the setup of Section \ref{section:prelim}, recall that the \textit{cohomological dimension} of the exact category \(\mathsf{Mod}_A\) is the supremum of projective dimensions of all \(A\)-modules. Likewise, we may define the cohomological bidimension of \(\mathsf{Mod}_A\), which is simply the cohomological dimension of \(\mathsf{Mod}_{A \otimes A^\op}\). For the group algebras \(A = \dvr[G]\) considered in this article, the cohomological (bi)dimension of the exact category \(\mathsf{Mod}_A\) is finite - this is well-known in the complex case and the proof adapts also to the nonarchimedean setting.

\begin{lemma}\label{lem:global-dim}
Let \(G\) be a word-hyperbolic group or a reductive \(p\)-adic group. Then \(\mathsf{Mod}_{\dvf[G]}\) has finite cohomological bidimension.
\end{lemma}

\begin{proof}
We first consider the word-hyperbolic case. Here the result of \cite{bader2013efficient}*{Theorem 1.7} works for integral coefficients, and the same proof works by changing the base ring to \(\dvf\). Now the same argument as in the proof of \cite{puschnigg2002kadison}*{Lemma 2.5} yields a \(G\)-equivariant chain map \(\theta \colon C_\bullet(G,\dvf) \to C_\bullet(G,\dvf)_R \subseteq C_\bullet(G,\dvf)\) that equals the identity in degree zero and vanishes in sufficiently high degrees. Here the number \(R > 6 \delta + 4\) is chosen such that the augmented Rips complex \(C_\bullet(G,\dvr)_R \to \dvr \to 0\) is contractible, and such an \(R\) always exists by \cite{bridson2013metric}*{Proposition 3.23, pp 469}). Note that the formula for \(\theta\) involves an anti-symmetrisation map, which introduces denominators, which is why we need to work over \(\dvf\). Now consider the operators \[\tilde{\nabla}_\bullet \defeq H_\bullet(\theta_\bullet, \mathrm{id}) \otimes \mathrm{id}_{\dvf[G]} \colon C_\bullet(G,\dvf) \otimes \dvf[G] \to C_{\bullet+1}(G,\dvf)_R \otimes \dvf[G],\] where we equip \(\dvf[G]\) with the inner conjugation action. Since these maps are \(G\)-equivariant, applying the functor \( - \hot_{\dvf[G]} \dvf\) yields linear maps \[\nabla_\bullet \colon \Omega^\bullet(\dvf[G]) \to \Omega^{\bullet + 1}(\dvf[G])\] on the Hochschild complex that satisfy \(\nabla_n \circ b_n + b_{n+1} \circ \nabla_n = \mathrm{id}_{\Omega^{n}(\dvf[G])}\) in degrees \(\geq n\) for some \(n >0\). Using the identification between the \(G\)-coinvariants of \(C_\bullet(G,\dvf) \otimes \dvf[G]\) and \(\Omega^\bullet(G)\), it can be checked that \(\nabla_n \colon \Omega^n(\dvf[G]) \to \Omega^{n+1}(\dvf[G])\) is an \(n\)-connection on \(\dvf[G]\) (see \cite{puschnigg2002kadison}*{Lemma 3.2}). Here the number \(n\) is a fixed number greater than the cardinality of the set \(\setgiven{g \in G}{l(g)<R}\). Finally, by \cite{Meyer:HLHA}*{Theorem A.123}, there is a finite length projective \(\dvf[G]\)-bimodule resolution of \(\dvf[G]\), as required.  

For the reductive \(p\)-adic case, the proof of \cite{meyer2006homological}*{Theorem 29} works integrally and goes through with obvious modifications. Indeed, the cellular complex \(C_\bullet(\mathcal{BT},\dvf)\) of the Bruhat-Tits building is a projective \(\dvf[G]\)-module resolution of the trivial representation of length equal to the dimension of the building \(\mathcal{BT}\). Equipping the modules \(C_\bullet(\mathcal{BT}, \dvf) \otimes \dvf[G]\) with the diagonal \(G\)-action, we get a finite-length projective \(\dvf[G]\)-bimodule resolution of \(\dvf[G]\).  
\end{proof}

\begin{corollary}\label{thm:combable-connection}
Let \(G\) be a \(\delta\)-hyperbolic group or a reductive \(p\)-adic group. Then there exists an \(n\)-connection on \(\dvr[G]^\dagger \otimes \dvf\) (respectively, \(\dvr[G]_K^\dagger \otimes \dvf\)), so that \[\mathsf{HP}(\dvr[G]^\dagger \otimes \dvf) \simeq X^{(n)}(\dvr[G]^\dagger \otimes \dvf),\] (respectively, \(\mathsf{HP}(\dvr[G]_K^\dagger \otimes \dvf) \simeq X^{(n)}(\dvr[G]_K^\dagger \otimes \dvf)\)) for some \(n\).
\end{corollary}

\begin{proof}
By Theorem \ref{theorem:main1} and \ref{main:isoco-p-adic}, the maps \(\dvr[G] \to \dvr[G]^\dagger\) and \(\dvr[G] \to \dvr[G]_K^\dagger\) are isocohomological. Since the composition of isocohomological embeddings is isocohomological, and the identity \(\dvf \to \dvf\) is obviously isocohomological, the maps \(\dvf[G] \to \dvr[G]^\dagger \otimes \dvf\) and \(\dvf[G] \to \dvr[G]_K^\dagger \otimes \dvf\) are isocohomological. By Lemma \ref{lem:global-dim}, the algebra \(\dvf[G]\) has finite bidimension. Corollary \ref{cor:HP-finite} now implies the result.
\end{proof}

\subsection{Cyclic homology of profinite group algebras}

When \(G\) is a reductive \(p\)-adic group, the algebra \(\dvr[G]_K^\dagger\)  only takes the topology of \(G\) into account in a rather restricted way,  namely through the growth of the length function. For a profinite group, or equivalently, a compact, Hausdorff totally disconnected group \(G = \varprojlim_i G_i\), a better algebra to look at is the Iwasawa algebra \[\Lambda(G) \defeq \varprojlim_{i} \dvr[G_i].\]  It is quite plausible that the canonical map \(\dvr[G] \to \Lambda(G)\) is isocohomological, but we prefer a direct computation of periodic cyclic homology for the Iwasawa algebra, using that the latter commutes with (countable) inverse limits (see below in Corollary \ref{cor:HP-Iwasawa}).

In what follows, we first consider the special case where \(G\) is a finite group, whence the group ring \(\dvr[G]\) is already a dagger algebra. Since \(\dvf\) is a field of characteristic zero by hypothesis, we have the following: 

\begin{theorem}\label{thm:HP-Iwasawa}
Let \(G\) be a finite group. Then \[\mathsf{HH}(\dvf[G]) \simeq \mathsf{HC}(\dvf[G]) \simeq \mathsf{HP}(\dvf[G]) \simeq X(\dvf[G]) \simeq \dvf[G]/[-,-].\] 
\end{theorem}

\begin{proof}
The characteristic zero assumption implies that the order of the group is invertible in \(\dvf\), so that \(\dvf[G]\) is a separable \(\dvf\)-algebra. The required equivalence now follows from \cite{Meyer:HLHA}*{A.131}.
\end{proof}


\begin{corollary}\label{cor:HP-Iwasawa}
Let \(G = \varprojlim_{n \in \N} G_n\) be a light profinite group, and let \(\Lambda(G) = ``\varprojlim " \dvr[G_n]\) be the Iwasawa algebra. Then \[\mathsf{HP}(\Lambda(G) \otimes \dvf) \simeq \varprojlim_n \mathsf{HP}(\dvf[G_n]) \simeq \varprojlim_n X(\dvf[G_n]) \simeq \varprojlim_n (\dvf[G_n]/[-,-]).\]  
\end{corollary}

\begin{proof}
The first equivalence follows from the fact that the completed projective tensor product on pro-bornologically torsionfree \(\dvr\)-modules commutes with projective limits (see for instance \cite{Meyer:HLHA}*{Theorem 4.37}). Now use Theorem \ref{thm:HP-Iwasawa}.
\end{proof}

\subsection{Cyclic homology of the \(p\)-adic noncommutative torus}

Next, we look at the nonarchimedean noncommutative torus defined in the previous section. It turns out that the algebra \(A = \dvr[\Z^2,c]\) has a very short (length \(2\)) free \(A\)-bimodule resolution given by

\[
0 \to A \otimes A \overset{b_2}\to A \otimes A \oplus A \otimes A \overset{b_1}\to A \otimes A \overset{b_0}\to A \to 0,
\] where \begin{gather*}
b_2(f \otimes g) = (\lambda f \otimes U_2 g - f U_2 \otimes g, -f \otimes U_1 g + \lambda f U_1 \otimes g) \\
b_1(f_1 \otimes f_2, f_3 \otimes f_4) = f_1 U_1 \otimes f_2 - f_1 \otimes U_2 f_2 + f_3 U_2 \otimes f_4 - f_3 \otimes U_2 f_4 \\ = f_1 (U_1 \otimes 1 - 1 \otimes U_2)f_2 + f_3(U_2 \otimes 1 - 1 \otimes U_1)f_4,\\
b_0(f \otimes g) = fg.
\end{gather*}   

In other words, the terms in the resolution are obtained by tensoring \(A \otimes A^\op\) with the Koszul resolution on \(\dvr^2\). Since tensoring by \(\dvf\) is exact, we still get a resolution 

\[0 \to \underline{A} \otimes \underline{A} \overset{b_2}\to \underline{A} \otimes \underline{A} \oplus \underline{A} \otimes \underline{A} \overset{b_1}\to \underline{A} \otimes \underline{A} \overset{b_0}\to \underline{A} \to 0,
\] where \(\underline{A} = \dvr[\Z^2, c] \otimes \dvf\). 

Since the canonical map \(\underline{A} \to \underline{A}^\dagger\) is an isocohomological embedding by Theorem \ref{thm:noncommutative-torus-isocohomological}, we deduce that \[0 \to \underline{A}^\dagger \otimes \underline{A}^\dagger \overset{b_2}\to \underline{A}^\dagger \otimes \underline{A}^\dagger \oplus \underline{A}^\dagger \otimes \underline{A}^\dagger \overset{b_1}\to \underline{A}^\dagger \otimes \underline{A}^\dagger \overset{b_0}\to \underline{A}^\dagger \to 0\] is an \(\underline{A}^\dagger\)-bimodule resolution of \(\underline{A}^\dagger\), where the boundary maps \(b_i\) are defined by the same formulas. We now use the resolution above to compute the Hochschild homology of the nonarchimedean noncommutative torus. 

\begin{theorem}\label{thm:Potdevin}
The natural map \(\underline{A} \to \underline{A}^\dagger\) induces the following isomorphisms  
\begin{equation}\label{eq:HH}
\mathrm{HH}_n(\underline{A}) \cong \mathrm{HH}_n(\underline{A}^\dagger) = 
\begin{cases}
\dvf \quad n = 0\\
 \dvf \oplus \dvf \quad n = 1 \\
 \dvf \quad n = 2 \\
 0 \quad n \geq 3
\end{cases}
\end{equation} of \(\dvf\)-vector spaces. 
\end{theorem}

\begin{proof}
To compute the Hochschild homology of \(\underline{A}\), we apply the functor \(- \otimes_{\underline{A} \otimes \underline{A}^\op} \underline{A}\) to a projective \(\underline{A}\)-bimodule resolution of \(\underline{A}\). For this, we pick the resolution \(P_\bullet(\underline{A}) = \underline{A} \otimes \underline{A}^\op \otimes \bigwedge^\bullet \dvf^2\) to get the complex 
\begin{equation}\label{eq:HH-2}
\mathsf{HH}(\underline{A}) \simeq [0 \to \underline{A} \overset{d_1}\to \underline{A} \oplus \underline{A} \overset{d_0}\to \underline{A} \to 0],
\end{equation} where \(d_1(a) = (\lambda U_2 a - a U_2, \lambda a U_1  - U_1 a)\) and \(d_0(a,b) = a U_1 - U_1 a + b U_2 - U_2 b\) for \(a\), \(b \in \underline{A}\). Now using that the map \(\underline{A} \to \underline{A}^\dagger\) is isocohomological by Theorem \ref{thm:noncommutative-torus-isocohomological}, we get that \( \underline{A}^\dagger \hot_{\underline{A}} P_\bullet(\underline{A}) \hot_{\underline{A}} \underline{A}^\dagger\) is a projective \(\underline{A}^\dagger\)-bimodule resolution of \(\underline{A}^\dagger\). Taking commutator quotients, we get the complex \begin{equation}\label{eq:HH-3}
\mathsf{HH}(\underline{A}^\dagger) \simeq [0 \to \underline{A}^\dagger \overset{d_1}\to \underline{A}^\dagger \oplus \underline{A}^\dagger \overset{d_0}\to \underline{A}^\dagger \to 0],
\end{equation} with the same formula for the differentials as in (\ref{eq:HH-2}). Furthermore, the complex in (\ref{eq:HH-2}) is the restriction along the natural map \(f \colon \underline{A} \to \underline{A}^\dagger\) of the complex in (\ref{eq:HH-3}), so that we have a chain map between complexes computing the Hochschild homologies of \(\underline{A}\) and \(\underline{A}^\dagger\).  

We now compute these Hochschild homology groups. Actually, we only compute  the Hochschild homology of the more complicated dagger completion \(\underline{A}^\dagger\), as the same computation applies to the purely algebraic case. Clearly the kernel of \(d_1\) is still \(\dvf\), yielding \(\mathrm{HH}_2(\underline{A}^\dagger)\). To compute \(\mathrm{HH}_1(\underline{A}^\dagger)\), let \((f,g) \in \ker(d_0)\). The power series \(f\) and \(g \in \underline{A}^\dagger\) are of the form \(f = \sum_{n,m \in \Z}^\infty a_{n,m} U^n V^m\) and \( g = \sum_{n,m \in \Z}^\infty b_{n,m} U^n V^m\), where \(\nu(a_{n,m}) + 1 \geq  C_f (\abs{n} + \abs{m})\) and \(\nu(b_{n,m}) + 1 \geq  C_g (\abs{n} + \abs{m})\) for \(C_f>0\) and \(C_g>0\). The condition \(fU - U f = gV - Vg\) and the relation \(UV = \lambda VU\) implies that \begin{equation}\label{HH-4}
 (\lambda^{n + 1} - 1)a_{m,n+1} = (1 -\lambda^{m+1})b_{m+1,n}
 \end{equation} for all \(n,m \in \Z\).   For the image of \(d_1\), take an \(f \in \underline{A}^\dagger\) and let \(d_1(f) = (f_1, f_2)\) for \(f = \sum_{n,m \in \Z} a_{m,n} U^m V^n\), and \(f_i = \sum_{n,m \in \Z} a_{m,n}^i U^m V^n\). A simple computation yields the relations 
 \begin{equation}\label{HH-5}
 a_{n,0}^1 = 0 = a_{0,m}^2, \quad (\lambda^m - 1)^{-1} a_{m,n+1}^1 = (1-\lambda^n)^{-1} a_{m+1,n}^2
 \end{equation} for all \(m\), \(n \in \Z\). The terms in \ref{HH-4} and \ref{HH-5} coincide, except the terms \(a_{0,n}\) and \(b_{m,0}\) which are zero in Equation \ref{HH-5} and are allowed to be arbitrary elements of \(\dvf\), yielding the required isomorphism \(\mathrm{HH}_1(\underline{A}^\dagger) \cong \dvf^2\). A similar computation yields that \(\HH_0(\underline{A}^\dagger) \cong \dvf\). \qedhere

\end{proof}

\begin{remark}
Note that Theorem \ref{thm:Potdevin} is a deviation from the archimedean case, where the Hochschild homologies of the algebraic and smooth noncommutative torus coincide if and only if \(\theta \in \R\) parameterising the noncommutative torus satisfies certain Diophantine approximation properties described in (\cite{connes1985non}). 
\end{remark}
As the natural map \(\underline{A} \to \underline{A}^\dagger\) induces an isomorphism in Hochschild homology, we have the following:

\begin{corollary}\label{cor:HP-torus}
We have isomorphisms in cyclic and periodic cyclic homology \[\mathrm{HC}_*(\underline{A}) \to \mathrm{HC}_*(\underline{A}^\dagger), \quad \mathrm{HP}_*(\underline{A}) \to \HP_*(\underline{A}^\dagger)\] for all \(* \in \N\).
\end{corollary}

In particular, the cyclic and periodic cyclic homology of the dagger-completed noncommutative torus simply reduces to the computations for the incomplete (algebraic) noncommutative torus in characteristic zero. The easiest way to compute the latter uses the Pimsner-Voiculescu sequence in bivariant \(K\)-theory. Here by bivariant \(K\)-theory, we mean the universal functor \(\mathsf{Alg}_{\dvr}^{\mathrm{tf}} \to \mathsf{kk}^\an\) from the category of torsionfree bornological \(\dvr\)-algebras to a triangulated category, satisfying homotopy invariance for dagger homotopies, stability for the \(\dvgen\)-adic completion of \(\mathbb{M}_\infty\), and excision for extensions of \(\dvr\)-algebras with a bounded linear section (see \cite{mukherjee2022nonarchimedean}).

\begin{theorem}\label{thm:HP-nc-torus}
Let \(A\) be a torsionfree bornological \(\dvr\)-algebra and \(\alpha \colon A \to A\) an algebra automorphism. Then there is a distinguished triangle \[\coma{\Omega} A \to A \overset{1 - \alpha}\to A \overset{\iota}\to A \rtimes_\alpha \Z\] in bivariant analytic \(K\)-theory, where \(\coma{\Omega}\) is the suspension used in \cite{mukherjee2022nonarchimedean} and \(\iota \colon A \to A \rtimes_{\alpha} \Z\) is the canonical inclusion. 
\end{theorem}
\begin{proof}
The proof of \cite{Cortinas-Thom:Bivariant_K}*{Theorem 7.4.1} goes through mutatis-mutandis. 
\end{proof}

Since the (bivariant) periodic cyclic homology complex \(\mathsf{HP}( - \otimes \dvf) \colon \mathsf{Alg}_{\dvr}^\tf \to \overleftarrow{\mathsf{Der}(\mathsf{Ind}(\mathsf{Ban}_\dvf))}\) on the category of torsionfree complete bornological \(\dvr\)-algebras satisfies the universal properties defining bivariant \(K\)-theory, we obtain a six-term exact sequence 
\[ 
\begin{tikzcd}
\HP_0(D,A) \arrow{r}{1 - \alpha_*} & \HP_0(D,A) \arrow{r}{\iota_*} & \HP_0(D, A \rtimes_\alpha \Z) \arrow{d}{} \\
\HP_1(D, A \rtimes_\alpha \Z) \arrow{u}{} & \HP_1(D,A) \arrow{l}{\iota_*} & \HP_1(D,A) \arrow{l}{1 - \alpha_*}
\end{tikzcd}
\] of \(\dvf\)-vector spaces, where \(D\) is an arbitrary torsionfree bornological \(\dvf\)-algebra. Of course, bivariant analytic cyclic homology is another invariant satisfying these properties, so the same exact sequence holds for \(\HA_*(D,-)\), \(* = 0,1\). Specialising to the \(\Z\)-action \(\alpha\) on \(A = \dvr[t,t^{-1}]\) defining the \(p\)-adic noncommutative torus, and setting \(D = \dvf\), we obtain the following:

\begin{corollary}\label{cor:HA-HP-nctorus}
For \(* = 0,1\), \[\HP_*(\underline{A}^\dagger) \cong \HP_*(\underline{A}) \cong \dvf^2 \cong \HA_*(\underline{A}) \cong \HA_*(\underline{A}^\dagger).\]
\end{corollary}

\end{document}